\documentclass[a4paper,english]{amsart}

\usepackage[T1]{fontenc}
\usepackage[utf8]{inputenc}
\usepackage{babel}
\usepackage{prettyref}
\usepackage{amsthm}
\usepackage{amssymb}
\usepackage[unicode=true,pdfusetitle,
 bookmarks=true,bookmarksnumbered=false,bookmarksopen=false,
 breaklinks=false,pdfborder={0 0 0},pdfborderstyle={},backref=false,colorlinks=false]
 {hyperref}

\makeatletter

\pdfpageheight\paperheight
\pdfpagewidth\paperwidth

\numberwithin{equation}{section}
\numberwithin{figure}{section}
  \theoremstyle{remark}
  \newtheorem*{acknowledgement*}{\protect\acknowledgementname}
\theoremstyle{plain}
\newtheorem{thm}{\protect\theoremname}[section]
  \theoremstyle{plain}
  \newtheorem{lem}[thm]{\protect\lemmaname}
  \theoremstyle{plain}
  \newtheorem{prop}[thm]{\protect\propositionname}
  \theoremstyle{definition}
  \newtheorem{defn}[thm]{\protect\definitionname}
  \theoremstyle{plain}
  \newtheorem{fact}[thm]{\protect\factname}
  \theoremstyle{remark}
  \newtheorem{rem}[thm]{\protect\remarkname}
  \theoremstyle{plain}
  \newtheorem{cor}[thm]{\protect\corollaryname}
  \theoremstyle{plain}
  \newtheorem{assumption}[thm]{\protect\assumptionname}

\usepackage{lmodern}
\usepackage[utf8]{inputenc}
\usepackage[all]{xy}

\newrefformat{rem}{Remark~\ref{#1}}
\newrefformat{assu}{Assumption~\ref{#1}}
\newrefformat{prop}{Proposition~\ref{#1}}
\newrefformat{lem}{Lemma~\ref{#1}}
\newrefformat{cor}{Corollary~\ref{#1}}
\newrefformat{fact}{Fact~\ref{#1}}
\newrefformat{def}{Definition~\ref{#1}}

\makeatother

  \providecommand{\acknowledgementname}{Acknowledgement}
  \providecommand{\assumptionname}{Assumption}
  \providecommand{\corollaryname}{Corollary}
  \providecommand{\definitionname}{Definition}
  \providecommand{\factname}{Fact}
  \providecommand{\lemmaname}{Lemma}
  \providecommand{\propositionname}{Proposition}
  \providecommand{\remarkname}{Remark}
\providecommand{\theoremname}{Theorem}

\begin{document}
\global\long\def\opname#1{\operatorname{#1}}

\global\long\def\C{\mathbb{C}}

\global\long\def\K{\mathbb{\mathbb{K}}}

\global\long\def\N{\mathbb{N}}

\global\long\def\Q{\mathbb{Q}}

\global\long\def\R{\mathbb{R}}

\global\long\def\Ralg{\mathbb{R^{\mathrm{alg}}}}

\global\long\def\T{\mathbb{T}}
\global\long\def\hA{\text{(A)}}
\global\long\def\hB{\text{(B)}}
\global\long\def\hC{\text{(C)}}

\global\long\def\Z{\mathbb{Z}}

\global\long\def\M{\mathfrak{M}}

\global\long\def\acl{\opname{acl}}

\global\long\def\dcl{\opname{dcl}}

\global\long\def\dom{\opname{dom}}

\global\long\def\img{\opname{Im}}

\global\long\def\bd#1{\delta#1}

\global\long\def\df{\opname{def}}

\global\long\def\interior{\opname{int}}

\global\long\def\cl{\opname{cl}}

\global\long\def\U{\mathcal{U}}

\global\long\def\V{\mathcal{V}}

\global\long\def\W{\mathcal{W}}

\global\long\def\Aut{\opname{Aut}}

\global\long\def\Inv{\opname{\M^{L}}}

\global\long\def\nin{\notin}

\global\long\def\eps{\varepsilon}

\global\long\def\h{\opname{{\bf h}}}

\global\long\def\ov#1{\overline{#1}}

\global\long\def\pr{\mathsf{p}}

\global\long\def\cal#1{\mathcal{\mathcal{\mathcal{{#1}}}}}

\global\long\def\monster{\mathcal{U}}

\global\long\def\suchthat{\,:\,}

\global\long\def\acl{\opname{acl}}

\global\long\def\Stab{\opname{Stab_{l}}}

\global\long\def\st{\opname{st}}

\global\long\def\St{\opname{St}}

\global\long\def\up#1{\pr^{-1}({#1})}

\global\long\def\pin{\pi_{n}}

\global\long\def\pind{\pi_{n}^{def}}

\global\long\def\pY{\pr^{-1}(O)}

\global\long\def\Y{O}

\title[A Vietoris mapping theorem for the homotopy of hyperdefinable sets]{A Vietoris-Smale mapping theorem for the homotopy of hyperdefinable
sets }

\author{Alessandro Achille}

\address{Computer Science department of UCLA, Vision Lab.}

\email{achille@cs.ucla.edu}

\author{Alessandro Berarducci}

\address{Università di Pisa, Dipartimento di Matematica, Largo Bruno Pontecorvo
5, 56127 Pisa, Italy}

\thanks{A.B. was partially supported by PRIN 2012 “Logica, Modelli e Insiemi”
and by Progetto di Ricerca d’Ateneo 2015 “Connessioni fra dinamica
olomorfa, teoria ergodica e logica matematica nei sistemi dinamici”.}

\email{alessandro.berarducci@unipi.it}

\date{3 June\,2017}


\begin{abstract}
Results of Smale (1957) and Dugundji (1969) allow to compare the homotopy
groups of two topological spaces $X$ and $Y$ whenever a map $f:X\to Y$
with strong connectivity conditions on the fibers is given. We apply
similar techniques in o-minimal expansions of fields to compare the
o-minimal homotopy of a definable set $X$ with the homotopy of some
of its bounded hyperdefinable quotients $X/E$. Under suitable assumption,
we show that $\pi_{n}(X)^{\df}\cong\pi_{n}(X/E)$ and $\dim(X)=\dim_{\R}(X/E)$.
As a special case, given a definably compact group, we obtain a new
proof of Pillay's group conjecture ``$\dim(G)=\dim_{\R}(G/G^{00}$)''
largely independent of the group structure of $G$. We also obtain
different proofs of various comparison results between classical and
o-minimal homotopy. 

\tableofcontents{}
\end{abstract}

\maketitle

\section{Introduction\label{sec:introduction}}

Let $M$ be a sufficiently saturated o-minimal expansion of a field.
We follow the usual convention in model theory \cite{Tent2012} to
work in a sufficiently saturated structure, so we assume that $M$
is $\kappa$-saturated and $\kappa$-strongly homogeneous for $\kappa$
a sufficiently big uncountable cardinal (this can always be achieved
going to an elementary extension). A set $X\subseteq M^{k}$ is \textbf{definable}
if it first-order definable with parameters from $M$, and it is \textbf{type-definable}
if it is the intersection of a \textbf{small} family of definable
sets, where ``small'' means ``of cardinality $<\kappa$''. The
dual notion of \textbf{$\bigvee$-definable} set is obtained by considering
unions instead of intersections. The hypothesis that $M$ has field
operations ensures that every definable set can be triangulated \cite{vdDries1998}. 

We recall that, given a definable group $G$, there is a normal type-definable
subgroup $G^{00}$, called \textbf{infinitesimal subgroup}, such that
$G/G^{00}$, with the \textbf{logic topology} \cite{Pillay2004a},
is a real Lie group \cite{Berarducci2005}. If in addition $G$ is
\textbf{definably compact} \cite{Peterzil1999}, we have $\dim(G)=\dim_{\R}(G/G^{00})$
\cite{Hrushovski2007a}, namely the o-minimal dimension of $G$ equals
the dimension of $G/G^{00}$ as a real Lie group. These results were
conjectured in \cite{Pillay2004a} and are still known as \textbf{Pillay's
conjectures}.

It was later proved that if $G$ is definably compact, then $G$ is
\textbf{compactly dominated }by $G/G^{00}$ \cite{Hrushovski2011a}.
This means that for every definable subset $D$ of $G$, the intersection
$\pr(D)\cap\pr(D^{\complement})$ has Haar measure zero (hence in
particular it has empty interior) where $\pr:G\to G/G^{00}$ is the
projection and $D^{\complement}$ is the complement of $D$. Special
cases were proved in \cite{Berarducci2004b} and \cite{Peterzil2007b}. 

The above results establish strong connections between definable groups
and real Lie groups. The proofs are complex and based on a reduction
to the abelian and semisimple cases, with the abelian case depending
in turn on the study of the fundamental group and on the counting
of torsion points \cite{Edmundo2004}. A series of result of P.~Simon
\cite{Simon2015,Simon2014a,Simon2013} provides however a new proof
of compact domination which does not rely on Pillay's conjectures
or the results of \cite{Edmundo2004}. More precisely, \cite{Simon2014a}
shows that fsg groups in o-minimal theories admit a smooth left-invariant
measure, and  \cite{Simon2015} contains a proof of compact domination
for definable groups admitting a smooth measure (even in a NIP context).
The fact that definably compact groups in o-minimal structures are
fsg is proved in \cite[Thm.\ 8.1]{Hrushovski2007a}. 

Our main theorem sheds new light on the connections between compact
domination and Pillay's conjectures, and concerns the topology of
certain hyperdefinable sets $X/E$, where $E$ is a bounded type-definable
equivalence relation on a definable set $X$. Under a suitable contractibility
assumption on the fibers of $\pr:X\to X/E$ (\ref{assumptionB}),
we obtain a homotopy comparison result between $X$ and $X/E$, and
in particular an isomorphism of homotopy groups 
\[
\pi_{n}(X)^{\df}\cong\pi_{n}(X/E)
\]
in the respective categories. Similar results apply locally, namely
replacing $X/E$ with an open subset $U\subseteq X/E$ and $X$ with
its preimage $\pr^{-1}(U)\subseteq X$, thus obtaining 
\[
\pi_{n}(\pr^{-1}(U))^{\df}\cong\pi_{n}(U).
\]
For the full result see \prettyref{thm:A} and \prettyref{thm:B}. 

From these local results and a form of ``topological compact domination''
(\ref{assumptionC}) we shall deduce that 
\[
\dim(X)=\dim_{\R}(X/E),
\]
namely the dimension of $X$ in the definable category equals the
dimension of $X/E$ in the topological category (\prettyref{thm:C}).
This yields a new proof of ``$\dim(G)=\dim_{\R}(G/G^{00})$'' for
compactly dominated groups which does not depend on the counting of
torsion points (for in fact it does not depend on the group structure!). 

Some comparison results between classical and o-minimal homotopy estabished
in \cite{Delfs1985,Berarducci2002c,Baro2009a} also follow (see \prettyref{cor:baro-otero}).
In particular, if $X=X(M)\subseteq M^{k}$ is a closed and bounded
$\emptyset$-semialgebraic and $\st:X\to X(\R)$ is the standard part
map, we can take $E=\ker(\st)$ and deduce
\[
\pi_{n}(X)^{\df}\cong\pi_{n}(X(\R)).
\]
This work can be considered as a continuation of the line of research
initiated \cite{Berarducci2011}: while in that paper we focused on
the fundamental group, here we manage to encompass the higher homotopy
groups and more generally homotopy classes $[X,Y]$ of map $f:X\to Y$
is the relevant categories. 

We have tried to make this paper as self-contained as possible. The
proofs of the homotopy results are somewhat long but elementary and
all the relevant notions are recalled as needed. The paper is organized
as follows. 

In Section 2 we recall the notions of definable space and definable
manifold, the main example being a definable group $G$. 

In Section 3 we introduce the logic topology on the quotient $X/E$
of a definable set $X$ by a bounded type-definable equivalence $E$. 

In Section 4 we recall the notion of ``normal triangulation'' due
to Baro \cite{Baro2010a}, and we show how to produce normal triangulations
satisfying some additional properties. 

In sections 5 and 6 we illustrate some of the analogies between the
standard part map and the map $G\to G/G^{00}$, where $G$ is a definably
compact group and $G/G^{00}$ has the logic topology. 

These analogies are further developed in Section 7, where we discuss
various versions of ``compact domination''. 

In Sections 8,9 we work in the category of classical topological spaces
and we establish a few results for which we could not find a suitable
reference. In particular in Section 8 we show that given an open subset
$U$ of a a triangulable space, any open covering of $U$ has a refinement
which is a good cover. 

In Section 10 we recall the definition of definable homotopy. 

Sections 11,12 and 13 contain the main results of the paper, labeled
Theorem A (\ref{thm:A}), Theorem B (\ref{thm:B}), and Theorem C
(\ref{thm:C}), respectively, as the titles of the corresponding sections. 

In Theorem A we prove that there is a homomorphism $\pi_{n}^{\df}(X)\to\pi_{n}(X/E)$
from the definable homotopy groups of $X$ and the homotopy groups
of $X/E$, under a suitable assumption on $E$. We actually obtain
a more general result of which this is a special case. 

In Theorem B we strengthen the assumptions to obtain an isomorphism:
$\pi_{n}^{\df}(X)\cong\pi_{n}(X/E)$. Since the standard part map
can put in the form $\pr:X\to X/E$ for a suitable $E$, some known
comparison results between classical and o-minimal homotopy will follow. 

Finally, in Theorem C we add the assumption of ``topological compact
domination'' to obtain $\dim(X)=\dim_{\R}(X/E)$ and we deduce $\dim(G)=\dim_{\R}(G/G^{00})$
and some related results. 
\begin{acknowledgement*}
Some of the results of this paper were presented at the 7th meeting
of the Lancashire Yorkshire Model Theory Seminar, held on December
the 5th 2015 in Preston. A.B.~wants to thank the organizers of the
meeting and acknowledge support from the Leverhulme Trust (VP2-2013-055)
during his visit to the UK. The results were also presented at the
Thematic Program On Model Theory, International Conference, June 20-24,
2016, University of Notre Dame. 
\end{acknowledgement*}

\section{Definable spaces}

A fundamental result in \cite{Pillay1988} establishes that every
definable group $G$ in $M$ has a unique group topology, called \textbf{t-topology},
making it into a a \textbf{definable manifold}. This means that $G$
has a finite cover $U_{1},\ldots,U_{m}$ by t-open sets and for each
$i\leq m$ there is a definable homeomorphism $g_{i}:U_{i}\to U_{i}'$
where $U_{i}'$ is an open subset of some cartesian power $M^{k}$
with the topology induced by the order of $M$. The collection $(g_{i}:U_{i}\to X_{i})_{i\leq m}$
is called an atlas and $g_{i}$ is called a local chart. 

Definable manifolds are special cases of \textbf{definable spaces}
\cite{vdDries1998}. The notion of definable spaces is defined through
local charts $g_{i}:U_{i}\to U_{i}'$, like definable manifolds, with
the difference that now $U_{i}'$ is an arbitrary definable subset
of $M^{k}$, not necessarily open. In particular every definable subset
$X$ of $M^{k}$, with the topology induced by the order, is a definable
space (with the trivial atlas consisting of a single local chart),
but not necessarily a definable manifold. 

We collect in this section a few results on definable spaces which
shall be needed in the sequel. They depend on the saturation assumptions
on $M$. The results are easy and well known to the experts but the
proofs are somewhat dispersed in the literature. 
\begin{lem}
\label{lem:intersection-open}Let $(A_{i}\suchthat i\in I)$ be a
small downward directed family of definable open subsets of a definable
space $X$ (where ``small'' means $|I|<\kappa$). Then $\bigcap_{i\in I}A_{i}$
is open. 
\end{lem}
\begin{proof}
Let $x\in\bigcap_{i\in I}A_{i}$ and fix a definable fundamental family
$(B_{\eps}\suchthat\eps>0)$ of neighbourhoods of $x$ decreasing
with $\eps$ (for example take $B_{\eps}$ to be the points of $X$
at distance $<\varepsilon$ from $x$ in a local chart). Since $A_{i}$
is open in $X$, there is $\eps_{i}>0$ such that $B_{\eps_{i}}\subseteq A_{i}$.
By saturation, we can find an $\eps>0$ in $M$ such that $\eps<\eps_{i}$
for each $i\in I$. It follows that $B_{\eps}\subseteq\bigcap_{i\in I}A_{i}$,
so $x$ is in the interior of the intersection. 
\end{proof}
\begin{lem}
\label{lem:intersection-closures}Let $(X_{i}\suchthat i\in I)$ be
a small downward directed family of definable subsets of a definable
space. Then $\ov{\bigcap_{i\in I}X_{i}}=\bigcap_{i\in I}\ov{X_{i}}$. 
\end{lem}
\begin{proof}
The inclusion ``$\subseteq$'' is trivial. For the ``$\supseteq$''
direction let $x\in\bigcap_{i}\ov X_{i}$ and suppose for a contradiction
that $x\nin\ov{\bigcap_{i\in I}X_{i}}$. Then there is an open neighbourhood
$U$ of $x$ disjoint from $\bigcap_{i\in I}X_{i}$. By saturation
there is $i\in I$ such that $U$ is disjoint from $X_{i}$, hence
$x\nin\ov X_{i}$, a contradiction.
\end{proof}
\begin{lem}
\label{lem:closure-interior}Let $(X_{i}\suchthat i\in I)$ be a small
downward directed family of definable subsets of the definable space
$X$. Suppose that $H:=\bigcap_{i\in I}X_{i}$ is clopen. Then for
every $i\in I$ there is $j\in I$ such that $\overline{X_{j}}\subseteq\interior(X_{i})$. 
\end{lem}
\begin{proof}
Fix $i\in I$. Since $H$ is open, $H\subseteq\interior(X_{i})$.
Using the fact that $H$ is also closed, we have $H=\ov H=\ov{\bigcap_{i\in I}X_{i}}=\bigcap_{i\in I}\ov{X_{i}}$
(by \prettyref{lem:intersection-closures}). The latter intersection
is included in $\interior(X_{i})$, hence by saturation there is $j\in I$
such that $\ov{X_{j}}\subseteq\interior(X_{i})$. 
\end{proof}

\section{Logic topology}

Let $X$ be a definable set and consider a type-definable equivalence
relation $E\subseteq X\times X$ of \textbf{bounded index} (namely
of index $<\kappa$) and put on $X/E$ the \textbf{logic topology}:
a subset $\Y\subseteq X/E$ is open if and only if its preimage in
$X$ is $\bigvee$-definable, or equivalently $C\subseteq X/E$ is
closed if and only if its preimage in $X$ is type-definable. This
makes $X/E$ into a compact Hausdorff space \cite{Pillay2004a}. We
collect here a few basic results, including some results from \cite{Pillay2004a,Berarducci2005},
which shall be needed later. 
\begin{prop}
\label{prop:image-definable-closed}For every definable set $C\subseteq X$,
$\pr(C)$ is closed in $X/E$. 
\end{prop}
\begin{proof}
By definition of logic topology, we need to show that $\pr^{-1}(\pr(C))$
is type-definable. By definition, $x$ belongs to $\pr^{-1}(\pr(C))$
if and only if $\exists y\in C\::\:xEy$. Since $E$ is type-definable,
$xEy$ is equivalent to a possibly infinite conjunction $\bigwedge_{i\in I}\varphi_{i}(x,y)$
of formulas over some small index set $I$, and we can assume that
every finite conjunction of the formulas $\varphi_{i}$ is implied
by a single $\varphi_{i}$. By saturation it follows that we can exchange
$\exists$ and $\bigwedge_{i}$, hence $\pr^{-1}(\pr(C))=\{x\suchthat\bigwedge_{i}\exists y\varphi_{i}(x,y)\}$,
a type-definable set. 
\end{proof}
\begin{prop}
Let $C\subseteq U\subseteq X/E$ with $U$ open and $C$ closed. Then
there is a definable set $D\subseteq X$ such that $\pr^{-1}(C)\subseteq D\subseteq\pr^{-1}(U)$. 
\end{prop}
\begin{proof}
This is an immediate consequence of the fact that if a type-definable
set is contained in a $\bigvee$-definable set, there is a definable
set between them.
\end{proof}
\begin{prop}
\label{prop:interior} Let $y\in X/E$ and let $D\subseteq X$ be
a definable set containing $\pr^{-1}(y)$. Then $y$ is in the interior
of $\pr(D)$. Moreover, there is an open neighbourhood $U$ of $y$
such that $\pr^{-1}(U)\subseteq D$. 
\end{prop}
\begin{proof}
By \prettyref{prop:image-definable-closed}, $Z=\pr(X\setminus D)$
is a closed set in $X/E$ which does not contain $y$. Hence the complement
$O$ of $Z$ is an open neighbourhood of $y$ contained in $\pr(D)$.
We have thus proved the first part. For the second part note that,
since $X/E$ is compact Hausdorff, it is in particular a normal topological
space. We can thus find a fundamental system of open neighbourhoods
$U_{i}$ of $y$ such that $\{y\}=\bigcap_{i}U_{i}=\bigcap_{i}\ov U_{i}$
Each $\pr^{-1}(\ov U_{i})$ is type-definable, and their intersections
is contained in the definable set $D$, so there is some $i\in\N$
such that $\pr^{-1}(\ov U_{i})\subseteq D$. 
\end{proof}
For our last set of propositions we assume that $X$ is a definable
space, possibly with a topology different from the one inherited from
its inclusion in $M^{k}$. 
\begin{prop}
\label{prop:pcontinuous}Assume that $X$ is a definable space and
each fiber of $\pr:X\to X/E$ is a downward directed intersection
of definable open subsets of $X$. Then $\pr:X\to X/E$ is continuous. 
\end{prop}
\begin{proof}
By \prettyref{lem:intersection-open} the preimage of any point is
open, hence the preimage of every set is open. 
\end{proof}
\begin{prop}
\label{prop:proj-of-closure}Assume $\pr:X\to X/E$ is continuous
and let $C$ be a definable subset of $X$. Then $\pr(\ov C)=\pr(C)$. 
\end{prop}
\begin{proof}
It suffices to observe that $\pr(\ov C)\subseteq\ov{\pr(C})\subseteq\pr(C)$
where the first inclusion holds because $\pr$ is continuous and the
second by \prettyref{prop:image-definable-closed}.
\end{proof}

\section{Triangulation theorems}

The triangulation theorem \cite{vdDries1998} is a powerful tool in
the study of o-minimal structures expanding a field. In this section
we review some of the relevant results and we prove a specific variation
of the normal triangulation theorem of \cite{Baro2010a} for simplexes
with real algebraic vertices. 

Simplicial complexes are defined as in \cite{vdDries1998}. They differ
from the classical notion because simplexes are open, in the sense
that they do not include their faces. As in \cite{vdDries1998}, the
vertices of a simplicial complex are concrete points, namely they
have coordinates in the given o-minimal structure $M$ (expanding
a field). More precisely, given $n+1$ affinely independent points
$a_{0},\ldots,a_{n}\in M^{k}$, the (open)\textbf{ $n$-simplex} $\sigma_{M}=(a_{0},\ldots,a_{n})\subseteq M^{k}$
determined by $a_{0},\ldots,a_{n}$ is the set of all linear combinations
$\sum_{i=0}^{n}\lambda_{i}a_{i}$ with $\lambda_{0}+\ldots+\lambda_{n}=1$
and $0<\lambda_{i}<1$ (with $\lambda_{i}\in M)$. If we go to a bigger
model $N\succeq M$, we write $\sigma_{N}$ for the set defined by
the same formulas but with $\lambda_{i}$ ranging in $N$. We omit
the subscript if there is no risk of ambiguity. A \textbf{closed simplex}
is defined similarly but with the weak inequalities $0\leq\lambda_{i}\leq1$.
In other words a closed simplex is the closure $\bar{\sigma}=\cl(\sigma)$
of a simplex $\sigma$, namely the union of a simplex and all its
faces. 

An \textbf{simplicial complex} is a finite collection $P$ of (open)
simplexes, with the property that for all $\sigma,\theta\in P$, $\ov{\sigma}\cap\ov{\theta}$
is either empty or the closure of some $\delta\in P$ (a common face
of the two simplexes). We shall say that $P$ is a \textbf{closed
simplicial complex} if whenever it contains a simplex it contains
all its faces. In this case we write $\ov P$ for the collection of
all closures $\bar{\sigma}$ of simplexes $\sigma$ of $P$ and we
call $\bar{\sigma}$ a closed simplex of $P$. 

The \textbf{geometrical realization} $|P|$ of a simplicial complex
$P$ is the union of its simplexes. 

We shall often assume that $P$ is defined over $\Ralg$, namely its
vertices have real algebraic coordinates, so that we can realize $P$
either in $M$ or in $\R$. In this case, we write $|P|_{M}$ or $|P|_{\R}$
for the geometric realization of $P$ in $M$ or $\R$ respectively.
Notice that a simplicial complex is closed if and only if its geometrical
realization is closed in the topology induced by the order of $M$. 

If $L\subseteq P$ is a subcomplex of $P$, and $\sigma\in P$, we
define $|\sigma_{|L}|_{\R}=\sigma\cap|L|_{\R}$ and $|\sigma_{|L}|_{M}=\sigma\cap|L|_{M}$.
To keep the notation uncluttered, we simply write $\sigma_{|L}$ when
the model is clear from the context.
\begin{defn}
A triangulation of a definable set $X\subseteq M^{m}$ is a pair $(P,\phi)$
consisting of a simplicial complex $P$ defined over $M$ and a definable
homeomorphism $\phi:|P|_{M}\to X$. We say that the triangulation
$\phi$ is compatible with a subset $S$ of $X$ if $S$ is the union
of the images of some of the simplexes of $P$. 
\end{defn}
\begin{fact}[o-minimal triangulation theorem \cite{vdDries1998}]
\label{fact:triangulation}Every definable set $X\subseteq M^{m}$
can be triangulated. Moreover, if $S_{1},\ldots,S_{l}$ are finitely
many definable subsets of $X$, there is a triangulation $\phi:|P|_{M}\to X$
compatible with $S_{1},\ldots,S_{l}$. 
\end{fact}
Now, suppose that we have a triangulation $\phi:|P|_{M}\to X$ and
we consider finitely many definable subsets $S_{1},\ldots,S_{l}$
of $X$. The triangulation theorem tells us that there is another
triangulation $\psi:|P'|_{M}\to X$ compatible with $S_{1},\ldots,S_{l}$,
but it does not say that we can choose $P'$ to be a subdivision of
$P$, thus in general $|P'|_{M}$ will be different from $|P|_{M}$.
This is going to be a problem if we want to preserve certain properties.
For instance suppose that $\phi$ is a definable homotopy (namely
its domain $|P|_{M}$ has the form $Z\times I$ where $I=[0,1]$).
The triangulation theorem does not ensure that $\psi$ can be taken
to be a definable homotopy as well. 

The ``normal triangulation theorem'' of Baro \cite{Baro2010a} is
a partial remedy to this defect: it ensures that we can indeed take
$P'$ to be a subdivision of $P$, hence in particular $|P'|_{M}=|P|_{M}$,
although $\psi$ will not in general be equal to $\phi$. The precise
statement is given below. It suffices to consider the special case
when $X=|P|$ and $\phi$ is the identity. 
\begin{defn}
Let $P$ be an (open) simplicial complex in $M^{m}$ and let $S_{1},\ldots,S_{l}$
be definable subsets of $|P|$. A \textbf{normal triangulation} of
$P$ is a triangulation $(P',\phi')$ of $|P|$ satisfying the following
conditions: 

\begin{enumerate}
\item $P'$ is a subdivision of $P$;
\item $(P',\phi')$ is compatible with the simplexes of $P$; 
\item for every $\tau\in P'$ and $\sigma\in P$, if $\tau\subseteq\sigma$
then $\phi'(\tau)\subseteq\sigma$. 
\end{enumerate}
\end{defn}
From (3) it follows that the restriction of $\phi'$ to a simplex
$\sigma\in P$ is a homeomorphism onto $\sigma$ and $\phi$ is definably
homotopic to the identity on $|P|$. 
\begin{fact}[Normal triangulation theorem \cite{Baro2010a}]
\label{fact:normal-triangulation}If $S_{1},\ldots,S_{l}$ are finitely
many definable subsets of $|P|$, there exists a normal triangulation
of $P$ compatible with $S_{1},\ldots,S_{l}$ . 
\end{fact}
Since we are particularly interested in triangulations where the vertices
of the simplicial complex have real algebraic coordinates, we prove
the following proposition, which guarantees that the normal triangulation
of a real algebraic simplicial complex can be also chosen to be real
algebraic.
\begin{prop}
\label{prop:algebraic-subdivision}Let $P$ be a simplicial complex
in $M^{k}$ defined over $\Ralg$ and let $L$ be a subdivision of
$P$. Then there is a subdivision $L'$ of $P$ such that: 

\begin{enumerate}
\item $L'$ is defined over $\Ralg$; 
\item there is a simplicial homeomorphism $\psi:|L|\to|L'|$ which fixes
all the vertices of $L$ with real algebraic coordinates.
\end{enumerate}
\end{prop}
\begin{proof}
Since $L$ is a subdivision of $P$, we have an inclusion of the zero-skeleta
$|P^{0}|\subseteq|L^{0}|\subseteq M^{k}$. For each $v\in|L^{0}|$,
let $v_{1},\ldots,v_{k}\in M$ be the coordinates of $v$. The idea
is that the combinatorial properties of the pair $(P,L)$ (namely
the properties invariant by isomorphisms of pairs of abstract complexes)
can be described, in the language of ordered fields, by a first order
condition $\varphi_{L,P}(\bar{x})$ on the coordinates $\bar{x}$
of the vertices. We then use the model completeness of the theory
of real closed fields to show that $\varphi_{L,P}(\bar{x})$ can be
satisfied in the real algebraic numbers. 

The details are as follows. For each $v\in|L^{0}|$ we introduce free
variables $x_{1}^{v},\ldots,x_{k}^{v}$ and let $\bar{x}^{v}$ be
the $k$-tuple $x_{1}^{v},\ldots,x_{k}^{v}$. Finally let $\bar{x}$
be the tuple consisting of all these variables $x_{i}^{v}$ as $v$
varies. We can express in a first order way the following conditions
on $\bar{x}$:

\begin{enumerate}
\item If $\sigma=(v_{0},\ldots,v_{n})\in L$, then $\sigma(\bar{x}):=(\bar{x}^{v_{0}},\ldots,\bar{x}^{v_{n}})$
is $n$-simplex, namely $\bar{x}^{v_{0}},\ldots,\bar{x}^{v_{n}}$
are affinely independent;
\item If $\sigma_{1}$ and $\sigma_{2}$ are open simplexes of $L$ with
common face $\tau$, then $\cl(\sigma_{1}(\bar{x}))\cap\cl(\sigma_{2}(\bar{x}))=\cl(\tau(\bar{x}))$;
\item If $\sigma_{1}$ and $\sigma_{2}$ are open simplexes of $L$ with
no face in common, then $\cl(\sigma_{1}(\bar{x}))\cap\cl(\sigma_{2}(\bar{x}))=\emptyset$;
\item If $\sigma\subseteq\tau$ with $\sigma\in L$ and $\tau\in P$, then
$\sigma(\ov x)\subseteq\tau(\ov x)$.
\end{enumerate}
These clauses express the fact that the collection $\sigma(\bar{x})$
as $\sigma$ varies in $L$ is a simplicial complex $L(\bar{x})$
(depending on the value of $\bar{x}$) isomorphic to $L$. Similarly
we can define $P(\bar{x})$ and express the fact that $L(\bar{x})$
is a subcomplex of $P(\bar{x})$. Our desired formula $\phi_{P,L}(\bar{x})$
is the conjunction of these clauses together with the conditions $x_{i}^{v}=v_{i}$
whenever $v_{i}$ is real algebraic. By definition $\phi_{P,L}(\bar{x})$
holds in $M$ if we evaluate each variable $x_{i}^{v}$ as the $i$-th
coordinate of the vector $v$. By the model completeness of the theory
of real closed fields, the formula can be satisfied by a tuple $\bar{a}$
of real algebraic numbers. The map sending each $v$ to $\bar{a}^{v}$
induces the desired isomorphism $\psi:L\to L'=L(\bar{a})$. 
\end{proof}
Later we shall need the following. 
\begin{prop}
\label{prop:normal-small}Let $P$ be a simplicial complex, let $X$
be a definable space and let $f:|P|_{M}\to X$ be a definable function.
Let $\V=\{V_{i}\suchthat i\in I\}$ be a small family of $\bigvee$-definable
sets $V_{i}\subseteq X$ whose union covers the image of $f$. Then
there is a subdivision $P'$ of $P$ and a normal triangulation $(P',\phi)$
of $P$ such that for every $\sigma\in P'$, $(f\circ\phi)(\sigma_{M})$
is contained in some $V_{i}$. Moreover, if $P$ is defined over $\Ralg$,
we can take $P'$ defined over $\Ralg$. 
\end{prop}
\begin{proof}
By saturation of $M$ there is a finite set $J\subseteq I$ such that
$\img(f)\subseteq\bigcup_{i\in J}V_{i}$. Again by saturation there
are definable subsets $U_{i}\subseteq V_{i}$ for $i\in J$ such that
$\img(f)\subseteq\bigcup_{i\in J}U_{i}$. By \prettyref{fact:normal-triangulation}
there is a subdivision $P'$ of $P$ and a normal triangulation $(P',\phi)$
of $P$ compatible with the definable sets $f^{-1}(U_{i})$, for $i\in J$.
Thus for $\sigma\in P'$, there is $i\in J$ such that $\phi(\sigma_{M})\subseteq f^{-1}(U_{i})$,
namely $(f\circ\phi)(\sigma_{M})\subseteq U_{i}$. 

The ``moreover'' part follows from \prettyref{prop:algebraic-subdivision}.
Indeed, if $P$ is over $\Ralg$, we first obtain $(P',\phi)$ as
above. If $P'$ is over $\Ralg$ we are done. Otherwise, we take a
subdivision $P''$ of $P$ over $\Ralg$ and a simplicial isomorphism
$\psi:P''\to P'$, and replace $(P',\phi)$ with $(P'',\phi\circ\psi)$. 
\end{proof}

\section{Standard part map\label{sec:Standard-part-map}}

Let $X=X(M)\subseteq M$ be a definable set and suppose $X\subseteq[-n,n]$
for some $n\in\N$. Then there is a map $\st:X\to\R$, called \textbf{standard
part}, which sends $a\in X$ to the unique $r\in\R$ satisfying the
same inequalities $p<x<q$ with $p,q\in\Q$. 

More generally, let $X$ be a definable subset of $M^{k}$ and assume
$X\subseteq[-n,n]^{k}$ for some $n\in\N$. We can then define $\st:X\to\R^{k}$
component-wise, namely $\st((a_{1},\ldots,a_{k})):=(\st(a_{1}),\ldots,\st(a_{k}))$. 

Now let 
\[
E:=\ker(\st)\subseteq X\times X
\]
be the type-definable equivalence relation induced by $\st$, namely
$aEb$ if and only if $\st(a)=\st(b)$. There is a natural bijection
$\st(X)\cong X/E$ sending $\st(a)$ to the class of $a$ modulo $E$,
so in particular $E$ has bounded index. The next two propositions
are probably well known but we include the proof for the reader's
convenience. 
\begin{prop}
The natural bijection $\st(a)\mapsto a/E$ is homeomorphism 
\[
\st(X)\cong X/E
\]
where $X/E$ has the logic topology and $\st(X)\subseteq\R^{k}$ has
the euclidean topology. 
\end{prop}
\begin{proof}
Every closed subset $C$ of $\R^{k}$ can be written as the intersection
$\bigcap_{i}C_{i}$ of a countable collection of closed $\emptyset$-semialgebraic
sets $C_{i}$ (where ``$\emptyset$-semialgebraic'' means ``semialgebraic
without parameters''). We then have $\st(a)\in C$ if and only $a\in\bigcap C_{i}(M)$.
This shows that the closed sets $C\subseteq\st(X)\subseteq\R^{k}$
in the euclidean topology correspond to the sets whose preimage in
$X$ is type-definable, and the proposition follows. 
\end{proof}
Thanks to the above result we can identify $\st:X\to\st(X)$ and $\pr:X\to X/E$
where $E=\ker(\st)$. The next proposition shows that these maps are
continuous. 
\begin{prop}
\label{prop:preimage-open}The preimage of any point of $y\in\st(X)$
under $\st:X\to\st(X)$ is open in $X\subseteq M^{k}$. In particular,
the standard part map is continuous (as the preimage of every subset
is open). 
\end{prop}
\begin{proof}
Let $a\in X$ and let $r:=\st(a)\in\R^{k}$. Then $\st^{-1}(r)=\bigcap_{n\in\N}\{b\in X\suchthat|b-r|<1/n\}$.
This is a small intersection of relatively open subsets of $X$, so
it is open in $X$ by \prettyref{lem:intersection-open}.
\end{proof}
\begin{rem}
If $X=X(M)\subseteq M^{k}$ is $\emptyset$-semialgebraic, we may
interpret the definining formula of $X$ in $\R$ and consider the
set $X(\R)\subseteq\R^{k}$ of real points of $X$. If we further
assume that $X$ is closed and bounded, then $X\subseteq[-n,n]^{k}$
for some $n\in\N$, so we can consider the standard part map $\st:X\to\R^{k}$.
It is easy to see that in this case $\st(X)$ coincides with $X(\R)$,
so we can write $X(\R)=\st(X)\cong X/E$. 
\end{rem}
Our next goal is to stuty the fibers of $\st:X\to X(\R)$. We need
the following.
\begin{defn}
\label{def:open-star}Given a simplicial complex $P$ and a point
$x\in|P|$ (not necessarily a vertex), the \textbf{open star} of $x$
with respect to $P$, denoted $\St(x,P)$, is the union of all the
simplexes of $P$ whose closure contains $x$. 
\end{defn}
\begin{prop}
\label{prop:intersection-of-stars}Given $x,y\in|P|$, if $\St(x,P)\cap\St(y,P)$
is non-empty, then there is $z\in|P|$ such that $\St(x,P)\cap\St(y,P)=\St(z,P)$. 
\end{prop}
\begin{proof}
Let $\sigma\in P$ be a simplex of minimal dimension included in $\St(x,P)\cap\St(y,P)$
and let $z\in\sigma$. We claim that $z$ is as desired. To this aim
it suffices to show that, given $\theta\in P$, we have $\theta\subseteq\St(x,P)\cap\St(y,P)$
if and only if $\theta\subseteq\St(z,P)$. 

For one direction assume $\theta\subseteq\St(x,P)\cap\St(y,P)$. Then
$\bar{\theta}\cap\bar{\sigma}$ is non-empty, as the intersection
contains both $x$ and $y$. It follows that there is a simplex $\delta\in P$
such that $\bar{\theta}\cap\bar{\sigma}=\bar{\delta}$. Notice that
$\delta$ is included in $\St(x,P)\cap\St(y,P)$ since its closure
contains $x$ and $y$. Since $\sigma$ was of minimal dimension contained
in this intersection, it follows that $\delta=\sigma$. But then $x\in\sigma\subseteq\bar{\theta}$,
hence $\theta\subseteq\St(x,P)$. 

For the other direction, assume $\theta\subseteq\St(z,P)$, namely
$z\in\bar{\theta}$. Since $z\in\sigma$, it follows that $\sigma\subseteq\bar{\theta}$
and thefore $\bar{\sigma}\subseteq\bar{\theta}$. But $x,y$ are contained
in $\bar{\sigma}$, so they are contained in $\bar{\theta}$, witnessing
the fact that $\theta\subseteq\St(x,P)\cap\St(y,P)$. 
\end{proof}
The following result depends on the local conic structure of definable
sets. 
\begin{prop}
\label{prop:standard-p-trivial}Let $X$ be closed and bounded $\emptyset$-semialgebraic
set and let $\st:X(M)\to\st(X)=X(\R)$ be the standard part map. Then
for every $y\in X(\R)$, the preimage $\st^{-1}(y)$ is an intersection
of a countable decreasing sequence $S_{0}\supseteq S_{1}\supseteq S_{2}\supseteq\ldots$
of definably contractible open subsets of $X$. 
\end{prop}
\begin{proof}
By the triangulation theorem (\prettyref{fact:triangulation}), there
is a simplicial complex $P$ over $\Ralg$ and a $\emptyset$-definable
homeomorphism $f:X\to|P|_{M}$. In this situation, $f_{\R}:X(\R)\to P(\R)$
is a homeomorphism and $\st(f(x))=f_{\R}(\st(x))$. Thus we can replace
$X$ with $P$ and assume that $X$ is the realization of a simplicial
complex. Therefore, we now have a closed simplicial complex $X(\R)$
over the reals, which is thus locally contractible. More precisely,
given $y\in X(\R)$ we can write $\{y\}$ as an intersection $\bigcap_{i\in\N}S_{i}$
where $S_{i}$ is the open star of $y$ with respect to the $i$-th
iterated barycentric subdivision of $P$. The preimage $\st^{-1}(y)$
can then be written as the corresponding intersection $\bigcap_{i\in\N}S_{i}(M)$
interpreted in $M$, and it now suffices to observe that each $S_{i}(M)$
is an open star (\prettyref{prop:intersection-of-stars}), hence it
is definably contractible (around any of its points). 
\end{proof}
Our next goal is to show that, much of what we said about the standard
part map, has a direct analogue in the context of definable groups,
with $\pr:G\to G/G^{00}$ in the role of the standard part. 

\section{Definable groups\label{sec:Type-definable-groups}}

Let $G$ be a definable group in $M$ and let $H<G$ be a type-definable
subgroup of bounded index. We may put on the coset space $G/H$ the
logic topology, thus obtaining a compact topological space. In this
context we have a direct generalization of \prettyref{prop:preimage-open}.
\begin{fact}[{\cite[Lemma 3.2]{Pillay2004a}}]
Every type-definable subgroup $H<G$ of bounded index is clopen in
the t-topology of $G$. In particular, the natural map $\pr:G\to G/H$
is continuous, where $G$ has the t-topology and the coset space $G/H$
has the logic topology. 
\end{fact}
If we further assume that $H$ is normal, then $G/H$ is a group and
we may ask whether the logic topology makes it into a topological
group. This is indeed the case \cite{Pillay2004a}. Some additional
work shows that in fact $G/H$ is a compact real Lie group \cite{Berarducci2005}.
In the same paper the authors show that $G$ admits a smallest type-definable
subgroup $H<G$ of bounded index (see \cite{Shelah2008} for a different
proof), which is denoted $G^{00}$ and called the \textbf{infinitesimal
subgroup}. When $G$ is \textbf{definably compact} in the sense of
\cite{Peterzil1999}, the natural map $\pr:G\to G/G^{00}$ shares
may of the properties of the standard part map. 
\begin{defn}
\label{def:ball}Let us recall that a definable set $B\subseteq X$
is called a \textbf{definable open ball} of dimension $n$ if $B$
is definably homeomorphic to $\{x\in M^{n}\suchthat|x|<1\}$; a \textbf{definable
closed ball }is defined similarly, using the weak inequality $\leq$;
we shall say that $B$ is a \textbf{definable proper ball }if there
is a definable homeomorphism $f$ from $\ov B$ to a definable closed
ball taking $\partial B$ to the definable sphere $S^{n-1}$.
\end{defn}
In analogy with \prettyref{prop:standard-p-trivial}, the following
holds. 
\begin{fact}
\label{fact:G00-contractible}\cite[Theorem 2.2]{Berarducci2009a}Let
$G$ be a definably compact group of dimension $n$ and put on $G$
the t-topology of \cite{Pillay1988}. Then there is a decreasing sequence
$S_{0}\supseteq S_{1}\supseteq S_{2}\supseteq\ldots$ of definably
contractible subsets of G such that $G^{00}=\bigcap_{i\in\N}S_{i}$. 
\end{fact}
The proof in \cite[Theorem 2.2]{Berarducci2009a} depends on compact
domination and the sets $S_{n}$ are taken to be ``cells'' in the
o-minimal sense. For later purposes we need the following strengthening
of the above fact, which does not present difficulties, but requires
a small argument. 
\begin{cor}
\label{cor:G00-balls}In \prettyref{fact:G00-contractible} we can
arrange so that, for each $i\in\N$ $\ov S_{i+1}\subseteq S_{i}$
and $S_{i}$ is a definable proper ball of dimension $n=\dim(G)$. 
\end{cor}
\begin{proof}
By \prettyref{lem:closure-interior} we can assume that $\ov S_{i+1}\subseteq S_{i}$
for every $i\in\N$. Since $M$ has field operations, a cell is definably
homeomorphic to a definable open ball (first show that it is definably
homeomorphic to a a product of intervals). In general it is not true
that a cell is a definable proper ball, even assuming that the cell
is bounded \cite{Berarducci2009c}. However by shrinking concentrically
$S_{i}$ via the homeomorphism, we can find a definable proper $n$-ball
$C_{i}$ with $\ov S_{i+1}\subseteq C_{i}\subseteq S_{i}$. To conclude,
it suffices to replace $S_{i}$ with the interior of $C_{i}$. 
\end{proof}

\section{Compact domination}

A deeper analogy between the standard part map and the projection
$\pr:G\to G/G^{00}$ is provided by \prettyref{fact:additive-measure}
and \prettyref{fact:compact-domination} below. 
\begin{fact}[{\cite[Cor.\ 4.4]{Berarducci2004b}}]
\label{fact:additive-measure}Let $X$ be a closed and bounded $\emptyset$-semialgebraic
subset of $M^{k}$ and let $D$ be a definable subset of $X$. Then
$\st(D)\cap\st(D^{\complement})\subseteq\R^{k}$ has Lebesgue measure
zero. 
\end{fact}
The above fact was used in \cite{Berarducci2004b} to introduce a
finitely additive measure on definable subsets of $[-n,n]^{k}\subseteq M^{k}$
($n\in\N$) by lifting the Lebesgue measure on $\R^{k}$ through the
standard part map. In the same paper it was conjectured that, reasoning
along similar lines, one could try to introduce a finitely additive
invariant measure on definably compact groups (the case of the torus
being already handled thanks to the above result). When \cite{Berarducci2004b}
was written, Pillay's conjectures from \cite{Pillay2004a} were still
open, and it was hoped that the measure approach could lead to a solution.
A first confirmation to the existence of invariant measures came from
\cite{Peterzil2007b}, but only for a limited class of definable group.
A deeper analysis lead to the existence of invariant measures in every
definable compact group \cite{Hrushovski2007a} and to the solution
of Pillay's conjectures, as discussed in the introduction. Finally,
the following far reaching result was obtained, which can be considered
as a direct analogue to \prettyref{fact:additive-measure}. 
\begin{fact}[\cite{Hrushovski2011a}]
\label{fact:compact-domination}Let $G$ be a definably compact group
and consider the projection $\pr:G\to G/G^{00}$. Then for every definable
set $D\subseteq G$, $\pr(D)\cap\pr(D^{\complement})$ has Haar measure
zero. 
\end{fact}
In the terminology introduced in \cite{Hrushovski2007a} the above
result can be described by saying that $G$ is \textbf{compactly dominated
}by $G/G^{00}$. Perhaps surprisingly, when the above result was obtained,
Pillay's conjectures had already been solved, so compact domination
did not actually play a role in its solution. In hindsight however,
as we show in the last part of this paper (\prettyref{sec:C}), compact
domination can in fact be used to prove ``$\dim(G)=\dim_{\R}(G/G^{00})$'',
as predicted by Pillay's conjectures (the content of Pillay's conjectures
also includes the statement that $G/G^{00}$ is a real Lie group). 

To prepare the ground, we introduce the following definition. In the
rest of the section $E$ is a type-definable equivalence relation
of bounded index on a definable set $X$. 
\begin{defn}
\label{def:top-comp-dom1}We say that $X$ is \textbf{topologically
compactly dominated} by $X/E$ if for every definable set $D\subseteq X$,
$\pr(D)\cap\pr(D^{\complement})$ has empty interior, where $\pr:X\to X/E$
is the projection.
\end{defn}
Since ``measure zero'' implies ``empty interior'', topological
compact domination holds both for the standard part map (taking $E=\ker(\st)$)
and for definably compact groups. 

Notice that \prettyref{def:top-comp-dom1} can be given for definable
sets in arbitrary theories, not necessarily o-minimal, so it is not
necessary that $X$ carries a topology. However in the o-minimal case
a simpler formulation can be given, as in \prettyref{cor:top-compact-dom2}
below. 

We first recall some definitions. Let $X$ be a definable space. We
say that a type-definable set $Z\subseteq X$ is \textbf{definably
connected} if it cannot be written as the union of two non-empty open
subsets which are relatively definable, where a relatively definable
subset of $Z$ is the intersection of $Z$ with a definable set. 

Following \cite{vdDries1998}, we distinguish between the frontier
and the boundary of a definable set, and we write $\partial D:=\ov D\setminus D$
for the \textbf{frontier}, and $\bd{D:=\ov D\setminus D^{\circ}}$
for the \textbf{boundary}, where $D^{\circ}$ is the interior. 

A basic result in o-minimal topology is that the dimension of the
frontier of $D$ is less than the dimension of $D$. Here we shall
however be concerned with the boundary, rather than the frontier. 
\begin{prop}
\label{prop:type-boundary}Let $X$ be a definable space. Assume that
$\pr$ is continuous and each fiber of $\pr:X\to X/E$ is definably
connected. Then for every definable set $D\subseteq X$, $\pr(D)\cap\pr(D^{\complement})=\pr(\bd{}D)$. 
\end{prop}
\begin{proof}
We prove $\pr(D)\cap\pr(D^{\complement})\subseteq\pr(\bd D)$. So
let $y\in\pr(D)\cap\pr(D^{\complement})$. If for a contradiction
$\pr^{-1}(y)\cap\bd D=\emptyset$, then $\pr^{-1}(y)\cap D^{\circ}$
and $\pr^{-1}(y)\cap(D^{\complement})^{\circ}$ are both non-empty.
Since they are relatively definable in $\pr^{-1}(y)$ and open, we
contradict the hypothesis that $\pr^{-1}(y)$ is definably connected.
The opposite inclusion is easy, using the fact that $\pr(A)=\pr(\ov A)$
(\prettyref{prop:proj-of-closure}).
\end{proof}
In the light of the above proposition, topological compact domination
takes the following form.
\begin{cor}
\label{cor:top-compact-dom2}Assume that $X$ is a definable space,
$\pr:X\to X/E$ is continuous, and each fiber of $\pr$ is definably
connected. Then $X$ is topologically compactly dominated by $X/E$
if and only if the image $\pr(Z)$ of any definable set $Z\subseteq X$
with empty interior, has empty interior. 
\end{cor}
\begin{proof}
Suppose that the image of every definable set with empty interior
has empty interior. Given a definable set $D\subseteq X$, we want
to show that $\pr(D)\cap\pr(D^{\complement})$ has empty interior.
This follows from the inclusion $\pr(D)\cap\pr(D^{\complement})\subseteq\pr(\delta D)$
(\prettyref{prop:type-boundary}) and the fact that $\delta D$ has
empty interior.

Conversely, assume topological compact domination and let $Z$ be
a definable subset of $X$ with empty interior. By \prettyref{prop:proj-of-closure},
$\pr(\delta Z)\subseteq\pr(\ov Z)\cap\pr(\ov{Z^{\complement}})=\pr(Z)\cap\pr(Z^{\complement})$,
so $\pr(\delta Z)$ has empty interior. 
\end{proof}

\section{Good covers}

By a \textbf{triangulable space} we mean a compact topological space
which is homeomorphic to a polyhedron, namely to the realization $|P|_{\R}$
of a closed finite simplicial complex over $\R$.

\begin{defn}
Let $\U$ be an open cover of a topological space $Y$. We say that
$\U$ is a \textbf{good cover} if every finite intersection $U_{1}\cap\ldots\cap U_{n}$
of open sets $U_{1},\ldots,U_{n}\in\U$ is contractible.
\end{defn}
Our aim is to show that open subsets of a triangulable spaces have
enough good covers. We are going to use barycentric subdivions holding
a subcomplex fixed, as defined in \cite[p.\ 90]{Munkres1984}. We
need the following observation. 
\begin{rem}
\label{rem:relative-barycentric}Let $P$ be a closed (finite) simplicial
complex and let $L$ be a closed subcomplex. Let $P_{i}$ be the $i$-th
barycentric subdivision of $P$ holding $L$ fixed. Then for every
real number $\eps>0$ there is $i\in\N$ such that for every closed
simplex $\bar{\sigma}$ of $P_{i}$, either $\bar{\sigma}$ has diameter
$<\eps$ or $\bar{\sigma}$ lies inside the $\eps$-neighbourhood
of some simplex of $L$. 
\end{rem}
\begin{defn}
Let $\U$ be an open cover of a topological space $Y$. Given $A\subseteq Y$
we recall that the star of $A$ with respect to $\U$, denoted $\St_{\U}(A)$,
is the union of all $U\in\U$ such that $U\cap A\neq\emptyset$. We
say that $\U$ \textbf{star refines }another cover $\V$ if for each
$U\in\U$ there is a $\V\in\V$ such that $\St_{\U}(U)\subseteq V$.
We define $\St(\U)$ to be the cover consisting of the sets $\St_{\U}(U)$
as $U$ ranges in $\U$. 
\end{defn}
\begin{prop}
\label{assu:pushforward-assumptions}Let $O$ be an open subset of
a triangulable space $Y$ (not necessarily a manifold). Then every
open cover $\V$ of $\Y$ has a locally finite refinement $\U$ which
is a good cover. 
\end{prop}
\begin{proof}
We can assume that $Y$ is the geometric realization $|P|$ (over
$\R$) of a finite simplicial complex $P$. 

Since $Y$ is a metric space, so is $O\subseteq Y$. In particular
$O$ is paracompact, and therefore $\V$ has a locally finite star-refinement
$\W\prec\V$. We plan to show that $O$ is the realization of an infinite
simplicial complex $L$ with the property that each closed simplex
of $L$ is contained in some element of $\W$. Granted this, by \prettyref{prop:intersection-of-stars}
we can take $\U$ to be the open cover consisting of the sets $\St(x,L)$
for $x\in O$. 

To begin with, note that we can write $O$ as the union $O=\bigcup_{n\in\N}C_{n}$
of an increasing sequence of compact sets in such a way that every
compact subset of $O$ belongs to $C_{n}$ for some $n$ (it suffices
to define $C_{n}$ as the set of points at distance $\geq1/n$ from
the frontier of $O$). 

Since $C_{0}$ is compact, by the Lebesgue number lemma there is some
$\eps_{0}>0$ such that every subset of $C_{0}$ of diameter $<\eps_{0}$
is contained in some element of $\W$. Now let $P_{0}$ be an iterated
barycentric sudivision of $P$ with simplexes of diameter $<\eps$
and let $L_{0}$ be the largest closed subcomplex of $P_{0}$ with
$|L_{0}|\subseteq C_{0}$. Notice that every closed simplex of $L_{0}$
is contained in some element of $\W$. 

Starting with $P_{0},L_{0}$ we shall define by induction a sequence
of subdivisions $P_{i}$ of $P=P_{0}$ and subcomplexes $L_{i}$ of
$L_{0}$. For concreteness, let us consider the case $i=1$. 

The complex $P_{1}$ will be of the form $P_{0}^{(n)}$, where $P_{0}^{(n)}$
is the $n$-th iterated barycentric subdivision of $P_{0}$ holding
the subcomplex $L_{0}$ fixed. To choose the value of $n$ we proceed
as follows. By the Lebesgue number lemma there is some $\eps_{1}>0$
with $\eps_{1}<\eps_{0}/2$ such that every closed subset of $C_{1}$
of diameter $<\eps_{1}$ is contained in some element of $\W$. By
taking a smaller value for $\eps_{1}$ if necessary, we can also assume
(by definition of $L_{0}$), that the closed $\eps_{1}$-neighbourhood
of any closed simplex $\bar{\sigma}$ of $L_{0}$ is contained in
some element of $\W$. By \prettyref{rem:relative-barycentric} there
is some $n_{0}$ such that for every $n\geq n_{0}$ and for every
closed simplex $\bar{\sigma}$ of $P_{0}^{(n)}$, either $\bar{\sigma}$
is contained in the $\eps_{1}$ neighbourhood of some $\lambda\in L_{0}$,
or the diameter of $\bar{\sigma}$ is less then $\eps_{1}$. In both
cases, if $\bar{\sigma}$ is included in $C_{1}$, then it is contained
in some element of $\W$. 

We now define $P_{1}=P_{0}^{(n)}$ and we let $L_{1}$ be the biggest
closed subcomplex of $P_{1}$ with $|L_{1}|\subseteq C_{1}$. The
crucial observation is that $L_{0}$ is a subcomplex of $L_{1}$,
since both are subcomplexes of $P_{1}$ and $|L_{0}|\subseteq|L_{1}|$. 

Having defined $P_{0},L_{0},\eps_{0},P_{1},L_{1},\eps_{1}$ we can
continue in the same fashion: given $P_{i},L_{i},\eps_{i}$ we define
$P_{i+1},L_{i+1},\eps_{i+1}$ in the same way we defined $P_{1},L_{1},\eps_{1}$
starting from $P_{0},L_{0},\eps_{0}$ and observe that $\eps_{n}\to0$
as $n\to\infty$. 

Since by construction each $L_{i}$ is a subcomplex of $L_{i+1}$,
we can consider the infinite simplicial complex $L:=\bigcup_{i\in\N}L_{i}$.
We claim that its geometrical realization is $O$. Granted the claim,
by construction each closed simplex of $L:=\bigcup_{i\in\N}L_{i}$
is contained in some $W\in\W$, and the proof is finished. 

To prove the claim notice that by construction $\bigcup_{i}L_{i}\subseteq O$.
To prove the equality we must show that $L_{i}$ is not too small.
Consider for instance $L_{1}$. We claim that if $x\in O$ is such
that its closed $\eps_{1}$-neighbourhood is contained in $C_{1}$,
then $x\in|L_{1}|$. Indeed, consider the (open) simplex $\sigma\in P_{1}$
containing $x$. Then either $\bar{\sigma}$ has diameter $<\eps_{1}$
or it is included in the $\eps_{1}$-neighbourood of $|L_{0}|$, and
in both cases $\bar{\sigma}$ is included in $|L_{1}|$. The same
argument applies for an arbirary $i\in\N$ instead of $i=1$ and immediately
implies the desired claim (since $\eps_{i}\to0$). 
\end{proof}

\section{Homotopy}

Recall that two continuous maps $f_{0},f_{1}:Z\to Y$ between topological
spaces are \textbf{homotopic} if there is a continuous function $H:Z\times[0,1]\to Y$
such that $H(z,0)=f_{0}(z)$ and $H(z,1)=f_{1}(z)$ for every $z\in Z$. 

Given base points $z_{0}\in Z$ and $y_{0}\in Y$ and a function $f:Z\to Y$,
we write $f:(Z,z_{0})\to(Y,y_{0})$ if $f$ sends $z_{0}$ to $y_{0}$.
Given an homotopy $H$ between two maps $f_{0},f_{1}:(Z,z_{0})\to(Y,y_{0})$
we say that $H$ is a homotopy relative to $z_{0}$ if $H(z_{0},t)=y_{0}$
for all $t\in I$, where $I=[0,1]$.
\begin{defn}
\label{def:pointed-homotopy-classes}If $Z$ and $Y$ are topological
spaces, we let $[Z,Y]$ denote the set of all homotopy classes of
continuous functions from $Z$ to $Y$. Given base points $z_{0}\in Z$
and $y_{0}\in Y$, we let $[(Z,z_{0}),(Y,y_{0})]$, or simply $[Z,Y]_{0}$,
denote the set of all homotopy classes of continuous functions $f:(Z,z_{0})\to(Y,y_{0})$
relative to $z_{0}$. The $n$-th homotopy group is defined as 
\[
\pi_{n}(Y):=[S^{n},Y]_{0}
\]
where $S^{n}$ is the $n$-th sphere and we put on $\pi_{n}(Y)$ the
usual group operation if $n>0$ (see \cite{Hatcher2002} for the details). 
\end{defn}
In the rest of this section we work in the classical category of topological
spaces and we give a sufficient condition for two maps to be homotopic.
Later we shall need to adapt the proofs to the definable category,
but with additional complications. 
\begin{defn}
Given a collection $\U$ of subsets of a set $O$ and two functions
$f,g:Z\to O$, we say that $f$ and $g$ are \textbf{$\U$-close}
if for any $z\in Z$ there is $U\in\U$ such that both $f(z)$ and
$g(z)$ are in $U$. 
\end{defn}
The following definition is adapted from \cite[Note 4]{Dugundji}.
\begin{defn}
\label{def:UP-small}Let $f:Z\to Y$ be a function between two sets
$Z$ and $Y$. Let $P$ be a collection of sets whose union $\bigcup P$
includes $Z$, and let $\U$ be a collection of subsets of $Y$. We
say that $f$ is \textbf{$(\U,P)$-small} if for every $\sigma\in P$
the image $f(\sigma\cap Z)$ is contained in some $U\in\U$. 
\end{defn}
\begin{lem}
\label{lem:good_cover_extension} Let $\U$ be a locally finite good
cover of a topological space $Y$ and let $L$ be a closed subcomplex
of a closed simplicial complex $P$ defined over $\R$. Let $f:|L\cup P^{(0)}|_{\R}\to Y$
be a $(\U,\ov P)$-small map (recall that $\ov P$ is the collection
of all closures of simplexes of $P$). Then $f$ can be extended to
a $(\U,\ov P)$-small map $f':|P|_{\R}\to Y$ with the property that,
for all $U\in\U$ and for every closed simplex $\bar{\sigma}$ of
$P$, if $f(\bar{\sigma}_{|L\cup P^{(0)}})\subseteq U$, then $f'(\bar{\sigma})\subseteq U$. 
\end{lem}
\begin{proof}
Reasoning by induction we can assume that $f'$ is already defined
on $|L\cup P^{(k)}|$ and we only need to extend it to $|L\cup P^{(k+1)}|$.
Let $\sigma\in P^{(k+1)}$. We can identify $\bar{\sigma}$ with the
cone over its boundary $\partial\sigma$, so that every point of $\bar{\sigma}$
is determined by a pair $(t,x)$ with $t\in[0,1]$ and $x\in\partial\sigma$.
Let $U_{1},\ldots,U_{n}$ be the elements of $\U$ containing $f'(\bar{\sigma}_{|L\cup P^{(k)}})$
(notice that $n>0$ by the inductive hypothesis), let $V$ be their
intersection, and let $\phi:[0,1]\times V\to V$ be a retraction of
$V$ to a point. We extend $f'$ to $\bar{\sigma}$ sending $(t,x)\in\bar{\sigma}$
to $\phi(t,f'(x))\in V$. Note that if $f'(\bar{\sigma}_{|L\cup P^{(k)}})\subseteq U\in\U$,
then $U$ is one of the $U_{i},$ and since by construction $f'(\bar{\sigma})\subseteq V=\bigcap_{i}U_{i}$,
we get $f'(\bar{\sigma})\subseteq U$. 
\end{proof}
\begin{prop}
\label{prop:close_homotopic} Let $\U$ be a locally finite good cover
of a topological space $Y$, let $P$ be a closed simplicial complex
and let $f,g:|P|_{\R}\to Y$ be two maps. Assume that $f$ and $g$
are $\U$-close. Then, $f$ and $g$ are homotopic. 
\end{prop}
\begin{proof}
Since $f$ and $g$ are $\U$-close, the family $\V=\{f^{-1}(U)\cap g^{-1}(U):U\in\U\}$
is an open cover of $|P|_{\R}$. By the Lebesgue number lemma (since
we work over $\R$) there is an iterated barycentric subdivision $P'$
of $P$ such that every closed simplex of $P'$ is contained in some
element of $\V$. Then, by construction, for every $\sigma\in P'$
there is $U\in\U$ such that $f(\bar{\sigma})$ and $g(\bar{\sigma})$
are contained in $U$.

Let now $I=[0,1]$ and consider the simplicial complex $P'\times I$
with the standard triangulation (as in \cite[p.\ 112, Proof of 2.10]{Hatcher2002}).
Consider the subcomplex $P'\times\{0,1\}$ of $P'\times I$ and note
that it contains the $0$-skeleton of $P'\times I$. Define $f\sqcup g:|P\times\{0,1\}|_{\R}=|P'\times\{0,1\}|_{\R}\to Y$
as the function which sends $(x,0)$ to $f(x)$ and $(x,1)$ to $g(x)$.
Note that $f\sqcup g$ is $(\U,\ov{P'\times I})$-small. Since $\U$
is a good cover, by \prettyref{lem:good_cover_extension} we can extend
it to a $(\U,\ov{P'\times I})$-small function $H:|P\times I|_{\R}\to Y$.
This map is a homotopy between $f$ and $g$.
\end{proof}

\section{Definable homotopies}

Given a definable set $Z$ and a $\bigvee$-definable set $Y$, we
say that a map $f:Z\to Y$ is definable if it takes values in a definable
subset $Y_{0}$ of $Y$ and is definable as a function from $Z$ to
$Y_{0}$. We can adapt \prettyref{def:pointed-homotopy-classes} to
the definable category as follows. 
\begin{defn}
\label{def:definable-homotopy-classes}If $Z$ is a definable space
and $Y$ is a $\bigvee$-definable set, we let $[Z,Y]^{\df}$ denote
the set of all equivalence classes of definable continuous maps $f$
from $Z$ to $Y$ modulo definable homotopies. Similarly we write
$[Z,Y]_{0}^{\df}$ when we work with pointed spaces and homotopies
relative to the base point $z_{0}\in Z$. The $n$-th o-minimal homotopy
group is defined as 
\[
\pi_{n}(Y)^{\df}:=[S^{n},Y]_{0}^{\df}
\]
where $S^{n}$ is the $n$-th sphere in $M$. If $n>0$ we put on
$\pi_{n}(Y)^{\df}$ a group operation in analogy with the classical
case. 
\end{defn}
In \cite{Berarducci2002c} it is proved that if $Y$ is a $\emptyset$-semialgebraic
set, $\pi_{1}(Y)^{\df}\cong\pi_{1}(Y(\R))$, so in particular $\pi_{1}(Y)^{\df}$
is finitely generated. This has been generalized to the higher homotopy
groups in \cite{Baro2010}. We shall later give a self-contained proof
of both results. By the same arguments we obtain the following result
of \cite{Berarducci2010d}: given a definably compact group $G$ there
is a natural isomorphism $\pi_{n}(G)^{\df}\cong\pi_{n}(G/G^{00})$.
The new proofs yield a stronger result: if $\pr:G\to G/G^{00}$ is
the projection, for every open subset $O$ of $G/G^{00}$, there is
an isomorphism $\pi_{n}(\pr^{-1}(O))^{\df}\cong\pi_{n}(O)$. This
was so far known for $n=1$ \cite{Berarducci2011}. Notice that $\pr^{-1}(O)$
is $\bigvee$-definable, whence the decision to consider $\bigvee$-definable
sets in \prettyref{def:definable-homotopy-classes}. With the new
approach we obtain additional functoriality properties and generalizations,
as it will be explained in the rest of the paper.

\section{Theorem A\label{sec:pushforward}}

As above, let $X=X(M)$ be a definable space, and let $E\subseteq X\times X$
be a definable equivalence relation of bounded index. In this section
we work under the following assumption. 
\begin{assumption}[Assumption A]
\label{assumptionA}$X/E$ is a triangulable topological space and
the natural map $\pr:X\to X/E$ is continuous.
\end{assumption}
The fact that $X/E$ is triangulable allows us to apply the results
of Section 8 regarding the existence of good covers. Note that the
continuity of $\pr$ is not a vacuous assumption because $X/E$ has
the logic topology, not the the quotient topology. By the results
in \prettyref{sec:Standard-part-map} and \prettyref{sec:Type-definable-groups}
the assumption is satisfied in the special case $X/E=G/G^{00}$ (where
$G$ is a definably compact group) and also when $X$ is a closed
and bounded $\emptyset$-semialgebraic set and $E=\ker(\st)$. 

We shall prove that there is a natural homomorphism 
\[
\pi_{n}^{\df}(X)\to\pi_{n}(X/E).
\]
This will be obtained as a consequence of a more general result concerning
homotopy classes. The following definition plays a crucial role in
the definition of the homomorphism, and exploits the analogies between
the projection $\pr:X\to X/E$ and the standard part map. 
\begin{defn}
\label{def:approximation}Let $O\subseteq X/E$ be an open subset.
Let $\U$ be an open cover of $\Y\subseteq X/E$, and let $P$ be
a closed simplicial complex defined over $\Ralg$. Consider a definable
continuous map $f:|P|_{M}\to\pY$ and let $\st:|P|_{M}\to|P|_{\R}$
be the standard part map. We say that a continuous map $f^{*}:|P|_{\R}\to\Y$
is an \textbf{$\U$-approximation} of $f$ if $\pr\circ f$ and $f^{*}\circ\st$
are $\U$-close, namely the two paths from the upper-left to the lower-right
corner of the following diagram represent maps which are $\U$-close.\[
\xymatrix@C+0.7cm@R+0.3cm{
|P|_{M} \ar[r]^-{\st} \ar[d]_{f}
&|P|_{\R} \ar[d]^{f^*}\\
\pr^{-1}(O) \ar[r]_-{\pr} & O
}
\]We say that $f$ is \textbf{$\U$-approximable} if it has a $\U$-approximation. 
\end{defn}
In general, given $f$ and $\U$, we cannot hope to find $f^{*}$
which is a $\U$-approximation of $f$. However we shall prove that,
given $\U$, every definable continous function $f$ is definably
homotopic to a $\U$-approximable map. 
\begin{defn}
Given a collection $\U$ of open subsets of $X/E$ let $\pr^{-1}(\U)$
be the collection of consisting of the $\bigvee$-definable open sets
$\pr^{-1}(U)\subseteq X$ as $U$ varies in $\U$. 
\end{defn}
Notice that $f:|P|_{M}\to X$ is $(\pr^{-1}(\U),P)$-small (\prettyref{def:UP-small})
if and only if $(\pr\circ f):|P|_{M}\to X/E$ is $(\U,P)$-small.
The next lemma shows that in this situation we can ignore the difference
between closed and open simplexes. Recall that $\ov P=\{\bar{\sigma}\suchthat\sigma\in P\}$.
We have: 
\begin{lem}
\label{lem:Pclosed}Let $\U$ be a collection of open subsets of $X/E$
and let $f:|P|_{M}\to X$ be a definable continuous map. Then $f$
is $(\pr^{-1}(\U),P)$-small if and only if it is $(\pr^{-1}(\U),\ov P)$-small. 
\end{lem}
\begin{proof}
Let $\sigma\in P$. Since $f$ is continous, $f(\bar{\sigma})\subseteq\ov{f(\sigma)}$
and by \prettyref{prop:proj-of-closure} we have $\pr(\ov{f(\sigma)})=\pr(f(\sigma))$,
so if $f$ is $(\pr^{-1}(\U),P)$-small, it is also $(\pr^{-1}(\U),\ov P)$-small. 
\end{proof}
The following lemma shows that small maps are approximable. 
\begin{lem}
\label{lem:existence-of-approximations}Let $\V$ be a locally finite
good open cover of $\Y$ and let $f:|P|_{M}\to\pY$ be a $(\pr^{-1}(\V),P)$-small
map. Then there exists a $\V$-approximation $f^{*}:|P|_{\R}\to\Y$
of $f$. 
\end{lem}
\begin{proof}
Define $f^{*}$ on the zero-skeleton of $P$ by $f^{*(0)}(\st(v))=\pr(f(v))$
for any vertex $v$ of $|P|_{M}$ (since $v$ has coordinates in $\Ralg$
we can identify $\st(v)\in|P|_{\R}$ with $v$). Since $f$ is $(\pr^{-1}(\V),P)$-small,
$f^{*(0)}$ is $(\V,P)$-small and therefore, by \prettyref{lem:good_cover_extension}
(and \prettyref{lem:Pclosed}), we can extend $f^{*(0)}$ to a $(\V,P)$-small
map $f^{*}:|P|_{\R}\to\Y$. We claim that $f^{*}$ is a $\V$-approximation
of $f$. Indeed, fix a point $z\in|P|_{M}$ and let $\sigma=\sigma_{M}\in P$
be a simplex containing $z$. Since $f$ is $(\pr^{-1}(\V),P)$-small,
there is a $V\in\V$ such that $\pr\circ f(\sigma)\subset V$, so
in particular $\pr\circ f(\sigma_{M}^{(0)})=f^{*(0)}(\sigma_{\R}^{(0)})\subseteq V$.
By \prettyref{lem:good_cover_extension}, we also have $f^{*}(\sigma_{\R})\subseteq V$.
Since $\st(\bar{\sigma}_{M})=\bar{\sigma}_{\R}$, both $\pr\circ f(z)$
and $f^{*}(\st(z))$ are in $V$.
\end{proof}
The next lemma shows that every map $f$ is homotopic to a small (hence
approximable) map $f'$.
\begin{lem}
\label{lem:existence-small-def} Let $O\subseteq X/E$ be an open
subset of $X/E$. Given a definable map $f:|P|_{M}\to\pY$ and an
open cover $\U$ of $\Y$, we can find a subdivision $P'$ of $P$
and a normal triangulation $(P',\phi)$ of $|P|$ such that $f'=f\circ\phi$
is $(\pr^{-1}(\U),P')$-small. Moreover if $P$ is defined over $\Ralg$,
we can take $P'$ defined over $\Ralg$. Notice that $f'$ is homotopic
to $f$ (as $\phi$ is homotopic to the identity). 
\end{lem}
\begin{proof}
By \prettyref{prop:normal-small}.
\end{proof}
\begin{lem}
\label{lem:same-approx-homotopic}Let $\U$ be a star-refinement of
a good cover of $\Y$. Any two $\U$-approximations of $f:|P|_{M}\to\pr^{-1}(\Y)$
are homotopic. 
\end{lem}
\begin{proof}
Let $f_{1}^{*}$ and $f_{2}^{*}$ be two $\U$-approximations of $f$.
Then and $f_{1}^{*}$ and $f_{2}^{*}$ are $\St(\U)$-close and since
$\U$ star-refines a good cover they are homotopic by \prettyref{prop:close_homotopic}. 
\end{proof}
We are now ready to state the main result of this section. 
\begin{thm}[Theorem A]
\label{thm:A}Assume \ref{assumptionA}. 
\begin{enumerate}
\item For each open set $O\subseteq X/E$, there is a map
\[
{\pr}_{\Y}^{P}:[|P|_{M},\pY]^{\df}\to[|P|_{\R},\Y].
\]
\item The maps ${\pr}_{\Y}^{P}$ are natural with respect to inclusions
of open sets. More precisely, let $U\subseteq V$ be open sets in
$X/E$. Then, we have the following commutative diagram:\[
\xymatrix@C+0.7cm@R+0.3cm{
[|P|_{M},\pr^{-1}(U)] \ar[r]^-{\pr_{U}^{P}} \ar[d]_{i_{\pr^{-1}(U)}}
& [|P|_{\R},U] \ar[d]^{i_{U}}\\
[|P|_{M},\pr^{-1}(V)] \ar[r]_-{\pr_{V}^{P}} & [|P|_{\R},V]
}
\]where the vertical arrows are induced by the inclusions.
\item By the triangulation theorem, the same statements continue to hold
if we replace everywhere $|P|$ by a $\emptyset$-semialgebraic set. 
\item The results remain valid replacing all the homotopy classes with their
pointed versions as in \prettyref{def:pointed-homotopy-classes} and
\prettyref{def:definable-homotopy-classes}. 
\item In particular, for all $O\subseteq X/E$ there is a natural map $\pr_{O}^{S^{n}}:\pi_{n}^{\df}(\pr^{-1}(O))\to\pi_{n}(O)$
which is in a group homomorphisms when $n>0$. When $O=X/E$ we obtain
a homomorphism $\pi_{n}^{\df}(X)\to\pi_{n}(X/E)$. 
\end{enumerate}
\end{thm}
In the rest of the section we fix a closed simplicial complex $P$
in $M$ defined over $\Ralg$ and we prove \prettyref{thm:A}. We
shall define a map ${\pr}_{\Y}^{P}:[|P|_{M},\pY]^{\df}\to[|P|_{\R},\Y]$
determined by the following property: if $\U$ is a star-refinement
of a good cover of $\Y$, $f$ is $(\pr^{-1}(\U),P)$-small and $f^{*}$
is a $\U$-approximation of $f$, then $\pr_{\Y}^{P}([f])=[f^{*}]$.
A word of caution is in order: we are not claiming that if $f$ is
$\U$-approximable and $f^{*}$ is an approximation of $f$, then
$\pr_{\Y}^{P}([f])=[f^{*}]$. We are only claiming that this will
be the case if $f$ is $(\pr^{-1}(\U),P)$-small, which is a stronger
property than being $\U$-approximable. The reason for the introduction
of this stronger property, is that we are not able to show that if
two definably homotopic maps are $\U$-approximable, then their approximations
are homotopic. We can do this only if the maps are $(\pr^{-1}(\U),P)$-small.
The formal definition is the following. 
\begin{defn}[Definition of the map $\pr_{\Y}^{P}$]
\label{def:defofp}Let $O\subseteq X/E$ be an open set. Let $\U$
be an open cover of $O$ which is a star-refinement of a good cover
and let $f:|P|_{M}\to\pr^{-1}(O)$ be a definable map. By \prettyref{lem:existence-small-def}
there is a subdivision $P'$ of $P$ and a normal triangulation $(P',\phi)$
of $|P|$ such that $f'=f\circ\phi$ is $(\pr^{-1}(\U),P')$-small.
By \prettyref{lem:existence-of-approximations} $f'$ has a $\U$-approximation
$f'^{*}$. We shall see (\prettyref{lem:well-defined} below) that
the homotopy class $[f'^{*}]$ does not depend on the choice of $P',\phi$
and $f'^{*}$, so we can define $\pr_{\Y}^{P}([f])=[f'^{*}]$. 
\end{defn}
To prove that the definition is sound we need the following. 
\begin{lem}
\label{lem:star-close}Let $f_{0,}f_{1}:|P|_{M}\to X$ be definable
maps and let $f_{0}^{*}$ and $f_{1}^{*}$ be $\U$-approximations
of $f_{0},f_{1}$ respectively. If $f_{0},f_{1}$ are $\pr^{-1}(\U)$-close,
then $f_{0}^{*}$ and $f_{1}^{*}$ are $\St(\U)$-close. 
\end{lem}
\begin{proof}
Let $y\in|P|_{\R}$ and let $x\in|P|_{M}$ be such that $\st(x)=y$.
By definition of approximation $f_{0}^{*}(y)$ is $\U$-close to $\pr(f_{0}(x))$,
which by hypothesis is $\U$-close to $\pr(f_{1}(x))$, which in turn
is $\U$-close to $f_{1}^{*}(y)$. We deduce that $f_{0}^{*}(y)$
is $\St(\U)$-close to $f_{1}^{*}(y)$.
\end{proof}
We can now finish the proof that \prettyref{def:defofp} is sound. 
\begin{lem}
\label{lem:well-defined}Let $\U$ be a star-refinement of a good
cover of $\Y$. Let $f_{0},f_{1}:|P|_{M}\to\pY$ be definably homotopic
definable continous maps and let $(P_{0},\phi_{0})$ and $(P_{1},\phi_{1})$
be two normal triangulations of $P$ such that $f_{0}\circ\phi_{0}$
is $(\pr^{-1}(\U),P_{0})$-small and $f_{1}\circ\phi_{1}$ is $(\pr^{-1}(\U),P_{1})$-small.
Now let $(f_{0}\circ\phi_{0})^{*}$ and $(f_{1}\circ\phi_{1})^{*}$
be $\U$-approximations of $f_{0}\circ\phi_{0}$ and $f_{1}\circ\phi_{1}$
respectively. Then $(f_{0}\circ\phi_{0})^{*}$ and $(f_{1}\circ\phi_{1})^{*}$
are homotopic.
\end{lem}
\begin{proof}
First note that $f_{0}\circ\phi_{0}$ and $f_{1}\circ\phi_{1}$ are
definably homotopic, because so are $f_{0}$ and $f_{1}$ and $\phi_{0}$,
$\phi_{1}$ are both definably homotopic to the identity. Let $H:|P\times I|_{M}\to\pY$
be a definable homotopy between $f_{0}\circ\phi_{0}=H_{0}$ and $f_{1}\circ\phi_{1}=H_{1}$.
Let $P'$ be a common refinement of $P_{0}$ and $P_{1}$ (for the
existence see for instance \cite[Cor.\ 9.5.8]{Mukherjee2015}). 

Now let $(T,\psi)$ be a normal triangulation of $P'\times I$ such
that $H\circ\psi$ is $(\pr^{-1}(\U),T)$-small. Notice that $T$
induces two subdivisions $P'_{0}$ and $P'_{1}$ of $P'$ such that
$(P'_{0}\times0)\cup(P'_{1}\times1)$ is a subcomplex of $T$. Notice
that both $f_{0}\circ\phi_{0}$ and $f_{1}\circ\phi_{1}$ are $(\pr^{-1}(\U),P')$-small,
because the smallness property is preserved by refinining the triangulations.
Moreover, the restriction of $\psi$ to the subcomplex $(P'_{0}\times0)\cup(P'_{1}\times1)$
induces two normal triangulations $(P'_{0},\psi_{0})$ and $(P'_{1},\psi_{1})$
of $P'$, namely $\psi_{0}(x)=y$ if and only if $\psi(x,0)=(y,0)$,
and similarly for $\psi_{1}$. By the properties of normal triangulations,
for each $\sigma\in P'$, we have $\psi_{0}(\sigma)=\sigma=\psi_{1}(\sigma)$,
so $f_{0}\circ\phi_{0}\circ\psi_{0}$ and $f_{1}\circ\phi_{1}\circ\psi_{1}$
are also $(\pr^{-1}(\U),P')$-small. 

Now let $(H\circ\psi)^{*}:|P\times I|_{\R}\to O$ be a $\U$-approximation
of $H\circ\psi$. Then $(H\circ\psi)^{*}$ is a homotopy between two
maps, which are easily seen to be $\U$-approximations of $f_{0}\circ\phi_{0}\circ\psi_{0}$
and $f_{1}\circ\phi_{1}\circ\psi_{1}$ (the two maps induced by $H\circ\psi$
by restriction), so we may call them $(f_{0}\circ\phi_{0}\circ\psi_{0})^{*}$
and $(f_{1}\circ\phi_{1}\circ\psi_{1})^{*}$ respectively. Since $\psi_{0}$
fixes the simplexes of $P'$ and $f_{0}\circ\phi_{0}$ is $(\pr^{-1}(\U),P')$-small,
we have that $f_{0}\circ\phi_{0}\circ\psi_{0}$ is $\U$-close to
$f_{0}\circ\phi_{0}$ (because any point of $|P|_{M}$ belongs to
some $\sigma\in P'$ which is mapped into some element of $\pr^{-1}(\U)$
by both maps). By \prettyref{lem:star-close} it follows that $(f_{0}\circ\phi_{0}\circ\psi_{0})^{*}$
is $\St(\U)$-close to $(f_{0}\circ\phi_{0})^{*}$ hence homotopic
to it. Similarly $(f_{1}\circ\phi_{1}\circ\psi_{1})^{*}$ is homotopic
to $(f_{1}\circ\phi_{1})^{*}$ and composing the homotopies we obtain
the desired result. 
\end{proof}
\begin{lem}
Points (1) and (2) of \prettyref{thm:A} hold. 
\end{lem}
\begin{proof}
We have already proved that $\pr_{\Y}^{P}$ is well defined and we
need to establish the naturality with respect to inclusions of open
sets $U\subseteq V\subseteq X/E$. Let $f:|P|_{M}\to\pr^{-1}(U)\subseteq\pr^{-1}(V)$
be a continuos definable map and notice that $C=\pr\circ f(|P|_{M})$
is a closed set. By \prettyref{thm:A}(1), there are open covers $\U$
of $U$ and $\V$ of $V$ which star-refine a good cover of $U$ and
$V$ respectively. We can further assume that $\V$ refines $\U\cup\{C^{\complement}\}$.
By \prettyref{lem:existence-small-def} there is a definable homeomorphism
$\psi:|P|_{M}\to|P|_{M}$ definably homotopic to the identity such
that $f':=f\circ\psi$ is $\V$-approximable (and clearly definably
homotopic to $f$). Since $\psi(|P|_{M})=|P|_{M}$, we have $\pr\circ f'(|P|_{M})=C$.
Let $f'^{*}:|P|_{\R}\to V$ be a $\V$-approximation of $f'$. Then
by definition $\pr_{V}^{P}([f])=[f'^{*}]$. Now fix some $x\in|P|_{M}$,
and using the definition of $\V$-approximation find $V'\in\V$ such
that both $(f'^{*}\circ\st)(x)\in V'$ and $(\pr\circ f')(x)\in V'$.
Notice that the latter implies that $V'$ cannot be contained in $C^{\complement}$,
hence it is contained in some element of $\U$. This shows that $f'^{*}$
has image contained in $U$ and is a $\U$-approximation of $f'$.
It follows that $i_{U}\circ\pr_{U}^{P}([f])=\pr_{V}^{P}\circ i_{\pr^{-1}(U)}([f])=[f'^{*}]$. 
\end{proof}
\begin{lem}
\prettyref{thm:A}(3) holds, namely we can work with $\emptyset$-semialgebraic
sets instead of simplicial complexes. 
\end{lem}
\begin{proof}
If $Z$ is a $\emptyset$-semialgebraic set, there is a $\emptyset$-definable
homeomorphism $f:|P|\to Z$ where  $P$ is a simplicial complex $P$
with real algebraic vertices. We have induced bijections $f_{M}^{*}:[Z(M),\pr^{-1}(U)]^{\df}\simeq[|P|_{M},\pr^{-1}(U)]^{\df}$
and $f_{\R}^{*}:[|Z|_{\R},U]\simeq[|P|_{\R},U]$. The results now
follows from the previous points of the theorem.
\end{proof}
\begin{lem}
\prettyref{thm:A}(4) holds, namely we can fix a base point and work
with relative homology. 
\end{lem}
\begin{proof}
It suffices to notice that all the constructions in the proofs can
equivalently be carried out for spaces with base points.
\end{proof}
\begin{lem}
\label{lem:pi-n-proof}\prettyref{thm:A}(5) holds, namely for any
open set $\Y\subseteq X/E$ there is a well defined group homomorphism
\[
{\pr}_{\Y}^{P}:\pi_{n}(\up O)^{\df}\to\pi_{n}(O).
\]
\end{lem}
\begin{proof}
We have already proved that there is a natural map ${\pr}_{\Y}^{S^{n}}:\pi_{n}(\up O)^{\df}\to\pi_{n}(O).$
We need to check that this map is a group homomorphism. To this end,
let $S^{n-1}$ be the equator of $S^{n}$. Recall that, given $[f],[g]\in\pi_{n}(\up O)^{\df}$,
where $f,g:S^{n}\to\up O$, the group operation $[f]*[g]$ is defined
as follows. Consider the natural map $\phi:S^{n}\to S^{n}/S^{n-1}=S^{n}\vee S^{n}$,
and let $[f]*[g]=[(f\vee g)\circ\phi]$, where $f\vee g$ maps the
first $S^{n}$ using $f$, and the second using $g$. A similar definition
also works for $\pi_{n}(O)$. Now, we have to check that $\pr_{O}^{S^{n}}([f]*[g])=\pr_{O}^{P}([f])*\pr_{O}^{P}([g])$.

By the triangulation theorem we can identify $S^{n}$ with the realization
of a simplicial complex $P$ defined over $\Ralg$ and, modulo homotopy
and taking a subdivision, we can assume that $f$ and $g$ are $(\pr^{-1}(\U),P)$-small
where $\U$ is an open cover of $O$ star-refinining a good cover.
Let $f^{*}$ and $g^{*}$ be $\U$-approximations of $f,g$ respectively,
so that $\pr_{O}^{S^{n}}([f])=[f^{*}]$ and $\pr_{O}^{S^{n}}([g])=[g^{*}]$
. Now it suffices to observe that $f^{*}\vee g^{*}$ is a $\U$-approximation
of $f\vee g$.
\end{proof}
The proof of \prettyref{thm:A} is now complete. 

\section{Theorem B\label{sec:pullback}}

In this section we work under the following strengthening of \ref{assumptionA}. 
\begin{assumption}
\label{assumptionB}$X/E$ is a triangulable topological space and
each fiber of $\pr:X\to X/E$ is the intersection of a decreasing
sequence of definably contractible open sets. 
\end{assumption}
By \prettyref{prop:pcontinuous} the assumption implies in particular
that $\pr$ is continuos, so we have indeed a strengthening of \ref{assumptionA}.
The above contractibility hypothesis was already exploited in \cite{Berarducci2011,Berarducci2009a,Berarducci2007b}
and is satisfied by the main examples discussed in \prettyref{sec:Standard-part-map}
and \prettyref{sec:Type-definable-groups}. 
\begin{thm}[Theorem B]
\label{thm:B}Assume that $\pr:X\to X/E$ satisfies \ref{assumptionB}.
Then the map $\pr_{\Y}^{P}:[|P|_{M},\pY]^{\df}\to[|P|_{\R},\Y]$ in
\prettyref{thm:A} is a bijection and similarly for pointed spaces.
Thus in particular $\pi_{n}(X)^{\df}\cong\pi_{n}(X/E)$ and more generally
we have a natural isomorphism $\pi_{n}(\pr^{-1}(O))^{\df}\cong\pi_{n}(O)$
for every open subset $O\subseteq X/E$. 
\end{thm}
Recall that if $X=X(M)\subseteq M^{k}$ is a closed and bounded $\emptyset$-semialgebraic
and $\st:X\to X(\R)$ is the standard part map, we can identify $\pr:X\to X/E$
with $\st:X\to X(\R)$ and deduce the following result of \cite{Baro2009a}.
\begin{cor}
\label{cor:baro-otero}If $X=X(M)\subseteq M^{k}$ is a closed and
bounded $\emptyset$-semialgebraic and $\st:X\to X(\R)$ is the standard
part map 
\[
\pi_{n}(X)^{\df}\cong\pi_{n}(X(\R))
\]
and similarly $[|P|_{M},X]^{\df}\cong[|P|_{\R},X(\R)]$. 
\end{cor}
In the rest of the section we prove \prettyref{thm:B}. The main difficulty
is the following. The homotopy properties of a space are essentially
captured by the nerve of a good cover, but unfortunately it is not
easy to establish a correspondence\emph{ }between good covers of $X/E$
in the topological category and good covers of $X$ in the definable
category. One can try to take the preimages $\pr^{-1}(U)$ in $X$
of the open sets $U$ belonging a good cover of $X/E$, but these
preimages are only $\bigvee$-definable, and if we approximate them
by definable sets, we loose some control on the intersections. We
shall show however, that we can perform these approximations with
a controlled loss of the amount of ``goodness'' of the covers. Granted
all this, the idea is to lift homotopies from $X/E$ to $X$, with
an approach similar to the one of \cite{Smale1957,Dugundji}, namely
we start with the restriction of the relevant maps to the $0$-skeleton,
and we go up in dimension. 

In the rest of the section fix an open set $O\subseteq X/E$. We need
the following. 
\begin{lem}
\label{lem:refinement-balls}Let $\V$ be an open cover of $\Y$.
Then there is a refinement $\W$ of $\V$ such that for every $W\in\W$
there is $V\in\V$ and a definably contractible definable set $B\subseteq X$
such that $\pr^{-1}(W)\subseteq B\subseteq\pr^{-1}(V)$. 
\end{lem}
\begin{proof}
Let $y\in O$. By our assumption $\pr^{-1}(y)$ is a decreasing intersection
$\bigcap_{i\in\N}B_{i}(y)$ of definably contractible definable sets
$B_{i}(y)$. Now let $V(y)\in\V$ contain $y$ and note that $\pr^{-1}(V(y))$
is a $\bigvee$-definable set containing $\pr^{-1}(y)=\bigcap_{i\in\N}B_{i}(y)$.
By logical compactness $B_{n}(y)\subseteq\pr^{-1}(V(y))$ for some
$n=n(y)\in\N$. By \prettyref{prop:interior} we can find an open
neighbourood $W(y)$ of $y$ with $\pr^{-1}(W(y))\subseteq B_{n}(y)$.
We can thus define $\W$ as the collection of all the sets $W(y)$
as $y$ varies in $O$. 
\end{proof}
\begin{cor}
\label{cor:refinement-for-extension}Let $\V$ be an open cover of
$\Y$. Then there is a refinement $\W$ of $\V$ with the following
property: every definable continous map $f:|\partial\sigma|_{M}\to\pr^{-1}(W)\subseteq\pY$
whose domain is the boundary of a definable  simplex and whose image
is contained in $\pr^{-1}(W)$ for some $W\in\W$, can be extended
to a definable continous map $F:|\bar{\sigma}|_{M}\to\pr^{-1}(\Y)$
on the whole closed simplex $|\bar{\sigma}|_{M}$ with image contained
in $\pr^{-1}(V)$ for some $V\in\V$. 
\end{cor}
\begin{proof}
Let $\V$ and $\W$ be as in \prettyref{lem:refinement-balls}. By
hypothesis, and by the property of $\W$, we have that $f(|\partial\sigma|_{M})\subseteq\pr^{-1}(W)\subseteq B\subseteq\pr^{-1}(V)$
for some definably contractible set $B$ and some $V\in\V$. Then,
$f$ can be extended to a definable map on $|\bar{\sigma}|_{M}$ with
image contained in $B\subseteq\pr^{-1}(V)$. 
\end{proof}
\begin{defn}
If $\W$ and $\V$ are as in \prettyref{cor:refinement-for-extension},
we say that $\W$ is \textbf{semi-good} within $\V$. 
\end{defn}
\begin{lem}
\label{lem:realization_extension}For any open cover $\U$ of $\Y$
and any $n\in\N$, there is a refinement $\W$ of $\U$ such that,
given a $n$-dimensional closed simplicial complex $P$, a closed
subcomplex $L$, and a $(\pr^{-1}(\W),P)$-small definable continous
map $f:|L\cup P^{(0)}|_{M}\to\pY,$ there is a $(\pr^{-1}(\U),P)$-small
definable continous map $F:|P|_{M}\to\pY$ extending $f$. 
\end{lem}
\begin{proof}
Reasoning by induction, it suffices to show that given $k<n$ and
an open cover $\U$ of $\Y$, there is a refinement $\W$ of $\U$
such that, given a $n$-dimensional closed simplicial complex $P$
and a $(\pr^{-1}(\W),P)$-small definable map $f:|L\cup P^{(k)}|_{M}\to\pY$,
there is a $(\pr^{-1}(\U),P)$-small definable map $F:|L\cup P^{(k+1)}|_{M}\to\pY$
extending $f$. 

To this aim, consider three open covers $\W\prec\V\prec\U$ of $\Y$
such that $\V$ is a star-refinement of $\U$ and and $\W\prec\V$
is semi-good within $\V$. Let $\sigma\in P^{(k+1)}$ be a $(k+1)$-dimensional
closed simplex such that $\bar{\sigma}$ is not included in the domain
of $f$. Since $|\partial\sigma|_{M}\subseteq|\sigma\cap P^{(k)}|_{M}\subseteq\dom(f)$
and $f$ is $(\pr^{-1}(\W),P)$-small, there is $W\in\W$ such that
$f(|\partial\sigma|_{M})\subseteq f(|\bar{\sigma}\cap P^{(k)}|_{M})\subseteq\pr^{-1}(W)$.
By the choice of $\W$, there is $V_{\sigma}\in\V$ such that we can
extend $f_{|\partial\sigma}$ to a map $F_{\sigma}:|\bar{\sigma}|_{M}\to\pr^{-1}(V_{\sigma})$
and define $F:|L\cup P^{(k+1)}|_{\R}\to\pY$ as the union of $f$
and the various $F_{\sigma}$ for $\sigma\in P^{(k+1)}$. 

It remains to prove that $F:|L\cup P^{(k+1)}|_{\R}\to\pY$ is $(\pr^{-1}(\U),P)$-small.
To this aim let $\tau\in P$ be any simplex. By our hypothesis there
is $W\in\W$ such that $f(|\bar{\tau}\cap P^{(k)}|_{M})\subseteq\pr^{-1}(W)$.
Now let $V\in\V$ contain $W$. By construction each face $\sigma$
of $\tau$ belonging to $L\cup P^{(k+1)}$ is mapped by $F$ into
$\pr^{-1}(V_{\sigma})$ for some $V_{\sigma}\in\V$. Moreover $V_{\sigma}$
intersects $W$, so it is included in $\St_{\V}(W)$. The latter depends
only on $\tau$ and not on $\sigma$ and is is contained in some $U\in\U$.
We have thus shown that $\bigcup_{\sigma}V_{\sigma}$ is contained
in some $U\in\U$, thus showing that $F$ is $(\pr^{-1}(\U),P)$-small.
\end{proof}
\begin{defn}
Let $\U$ be an open cover of $\Y$. If $\W$ is as in \prettyref{lem:realization_extension}
we say that $\W$ is \textbf{$n$-good} within $\U$. If the only
member of $\U$ is $O$ (or if the choice of $\U$ is irrelevant),
we simply say that $\W$ is $n$-good. 
\end{defn}
\begin{lem}
\label{lem:close-def-homotopic}Let $n\in\N$ and let $\W$ be an
$n+1$-good cover of $O$. If $P$ is an $n$-dimensional simplicial
complex and $f,g:|P|_{M}\to\pY$ are definable continous functions
such that for every $\sigma\in P$ there is $W\in\W$ such that $f(\sigma)$
and $g(\sigma)$ are contained in $\up W$, then $f$ and $g$ are
definably homotopic. 
\end{lem}
\begin{proof}
Let $I=[0,1]$ and consider the simplicial complex $P\times I$ (of
dimension $n+1$) with the standard triangulation (as in \cite[p.\ 112, Proof of 2.10]{Hatcher2002}).
Consider the subcomplex $P\times\{0,1\}$ of $P\times I$ and note
that it contains the $0$-skeleton of $P\times I$. Define $f\sqcup g:|P\times\{0,1\}|_{M}\to O$
as the function which sends $(x,0)$ to $f(x)$ and $(x,1)$ to $g(x)$.
Note that $f\sqcup g$ is $(\pr^{-1}(\W),P\times I)$-small by hypothesis.
By \prettyref{lem:realization_extension} we can extend it to definable
continuos function $H:|P\times I|_{M}\to\pY$. This map is a homotopy
between $f$ and $g$.
\end{proof}
\begin{lem}
\label{lem:same-approximation}Let $n\in\N$. Let $\V$ be an open
covering of $O$ which is a star refinement of a $n+1$-good cover
$\W$. Given an $n$-dimensional simplicial complex $P$ and definable
continuos maps $f,g:|P|_{M}\to\pY$, if there is a map $f^{*}:|P|_{\R}\to O$
which is a $\V$-approximations of both $f$ and $g$, then $f$ and
$g$ are definably homotopic. 
\end{lem}
\begin{proof}
Let $P'$ be an iterated barycentric subdivision of $P$ such that
for each $\sigma\in P'$ there is $V\in\V$ such that $f^{*}(\bar{\sigma})\subseteq V$.
We claim that for each $\sigma\in P'$, there is a $W\in\W$ such
that $\pr\circ f(\sigma)$, $\pr\circ g(\sigma)$ (and $f^{*}\circ\st(\sigma)$)
are in $W$. Given this claim, we can conclude using \prettyref{lem:close-def-homotopic}.

To prove the claim, fix a $\sigma\in P'$ and let $V\in\V$ be such
that $f^{*}(\sigma)\subseteq V$. Since $f^{*}$ is $\V$-approximation
of $f$, for each $x\in\sigma$ there is $V_{x}\in\St(\V)$ such that
$\pr\circ f(x)$ and $f^{*}\circ\st(x)$ are in $V_{x}$, and similarly
there is a $V_{x}'$ such that $\pr\circ g(x)$ and $f^{*}\circ\st(x)$
are in $V_{x}'$. Since $V$ intersects both $V_{x}$ and $V_{x}'$,
$\St_{\V}(V)$ contains both $\pr\circ f(x)$ and $\pr\circ g(x)$,
and since $\St(\V)$ refines $\W$, there is $W\in\W$ with the same
property.
\end{proof}
\begin{lem}
\label{lem:pr-injective}Let $n\in\N$. There is an open cover $\W$
of $\Y$ such that, given an $n$-dimensional simplicial complex $P$
and definable continuous maps $f,g:|P|_{M}\to\pY$, if $f^{*}$ and
$g^{*}$ are $\W$-approximations of $f$ and $g$ respectively, and
$G:|P\times I|_{\R}\to\Y$ is a homotopy between $f^{*}$ and $g^{*}$,
then there is a definable homotopy $H:|P\times I|_{M}\to\pY$ between
$f$ and $g$. 
\end{lem}
\begin{proof}
Let $\U$ be an $n+1$-good covering of $O$, let $\V$ be such that
$\St(\V)$ is a star refinement of $\U$ and let $\W\prec\V$ be $n+1$-good
within $\V$. Let $T$ be a barycentric subdivision of $|P\times I|_{\R}$
such that $G$ is $(\W,T)$-small. Let $H^{(0)}:T^{(0)}\to\pY$ be
such that $\pr\circ H^{(0)}=G\circ\st$ on the vertices of $T$. For
each simplex $\sigma\in T$ there is $W\in\W$ such that $G(|\bar{\sigma}|_{\R})\subseteq W$,
hence $H^{(0)}(|\sigma^{(0)}|_{M})\subseteq W$. Using \prettyref{lem:realization_extension}
we can extend $H^{0}$ to a $(\pr^{-1}(\V),T)$-small definable continous
map $H:|T|_{M}\to\pY$. If $\mathbf{x}=(x,0)\in|P\times0|_{M}$ is
a vertex of $T$, then $(f^{*}\circ\st)(x)=(G\circ\st)(\mathbf{x})=(\pr\circ H)(\mathbf{x})$
by construction. Since moreover $f^{*}\circ\st$ and $\pr\circ H$
are $(\V,T)$-small, it follows that $f^{*}\circ\st$ and $\pr\circ H_{|0}$
are $\St(\V)$-close, hence $f^{*}$ is a $\St(\V)$-approximation
of both $f$ (by hypothesis) and of $H_{|0}$. We can then conclude
using \prettyref{lem:same-approximation} that $f$ and $H_{|0}$
are homotopic, and, similarly, that $H_{|1}$ is homotopic to $g$.
Composing the homotopies, we can finally prove that $f$ is homotopic
to $g$. 
\end{proof}
\begin{lem}
\label{lem:pr-surjective}Let $\U$ be an open cover of $\Y$. Let
$f^{*}:|P|_{\R}\to\Y$ be a continuous map. Then, we can find a map
$f:|P|_{M}\to\pY$ such that $f^{*}$ is a $\U$-approximation of
$f$.
\end{lem}
\begin{proof}
Let $n=\dim(P)$, let $\V$ be a star-refinement of $\U$ and let
$\W$ be $n$-good within $\V$. Consider an iterated baricentric
subdivision $P'$ of $P$ such that $f^{*}$ is $(\W,P')$-small.
Let $f^{(0)}:P'^{(0)}\to\pY$ be such that $\pr\circ f^{(0)}=f^{*}\circ\st$
on the vertices of $P'$. Then we can apply \prettyref{lem:realization_extension}
to extend $f^{(0)}$ to a $(\V,P')$-small map $f:|P|_{M}\to\pY$.
Now notice that $\pr\circ f$ and $f^{*}\circ\st$ are $\St(\V)$-close
(since they are $(\V,P')$-small and they coincide on the vertices),
and therefore $f^{*}$ is a $\St(\V)$-approximation of $f$, so also
a $\U$-approximation.
\end{proof}
We can now finish the proof of the main result of this section. 
\begin{proof}[Proof of \prettyref{thm:B}]
First we prove the injectivity. Suppose that $\pr_{\Y}^{P}([f])=\pr_{\Y}^{P}([g])$.
Let $\W$ be as in \prettyref{lem:pr-injective}. Choosing a different
representative of the homotopy classes we can assume without loss
of generality that $f$ and $g$ are $(\pr^{-1}(\W),P)$-small, $\pr_{\Y}^{P}([f])=[f^{*}]$
and $\pr_{\Y}^{P}([g])=[g^{*}]$, where $f^{*}$ and $g^{*}$ are
$\W$-approximation of $f$ and $g$ respectively. By definition $[f^{*}]=[g^{*}]$,
that is $f^{*}$ and $g^{*}$ are homotopic. We can now apply \prettyref{lem:pr-injective}
to find a definable homotopy between $f$ and $g$, and so $[f]=[g]$.

The surjectivity is immediate from \prettyref{lem:pr-surjective}.
\end{proof}

\section{Theorem C\label{sec:C}}

In this section we work under the following strengthening of \ref{assumptionB},
where we considers definable proper balls (\prettyref{def:ball})
instead of definably contractible sets. 
\begin{assumption}
\label{assumptionB-1}$X/E$ is a triangulable manifold, $X$ is a
definable manifold, and each fiber of $\pr:X\to X/E$ is the intersection
of a decreasing sequence of definable proper balls. 
\end{assumption}
We also need:
\begin{assumption}[Topological compact domination]
\label{assumptionC}The image under $\pr:X\to X/E$ of a definable
subset of $X$ with empty interior, has empty interior. 
\end{assumption}
Both assumptions are satisfied by $\pr:G\to G/G^{00}$ for any definably
compact group $G$ (see section 7 and \prettyref{cor:G00-balls}). 
\begin{thm}[Theorem C]
\label{thm:C}Under Assumption \ref{assumptionC}, we have $\dim(X)=\dim_{\R}(X/E).$ 
\end{thm}
To prove the theorem the idea is to exploit the following link between
homotopy and dimension: given a manifold $Y$ and a punctured open
ball $U:=A\setminus\{y\}$ in $Y$, the dimension of $Y$ is the least
integer $i$ such that $\pi_{i-1}(U)\neq0$. 
\begin{prop}
\label{prop:dimX-geq-dimXE}$\dim(X)\geq\dim_{\R}(X/E)$. 
\end{prop}
\begin{proof}
Let $n=\dim(X)$ and $N=\dim_{\R}(X/E)$. Fix $x\in X$ and let $y=\pr(x)$.
Let $B_{0}$ be an open definable ball containing $\pr^{-1}(y)$.
Since $X/E$ is a manifold, there is a decreasing sequence of proper
balls $A_{i}\subseteq X/E$ such that $y=\bigcap_{i\in\N}A_{i}=\bigcap_{i\in\N}\ov A_{i}$.
Now $B_{0}\supseteq\pr^{-1}(y)=\bigcap_{i\in I}\pr^{-1}(\ov{A_{i}})$
and $\pr^{-1}(\ov{A_{i}})$ is type-definable (because $\ov A_{i}$
is closed), so there is some $i\in\N$ with $\pr^{-1}(\ov{A_{i}})\subseteq B_{0}$.
Let $A=A_{i}$ and observe that $\pr^{-1}(A)$ is $\bigvee$-definable
and contains the type definable set $\pr^{-1}(y)$. Since the latter
is a decreasiong intersection of definable proper balls, there is
some definable proper ball $B_{1}$ such that
\[
x\in\pr^{-1}(y)\subseteq B_{1}\subseteq\ov B_{1}\subseteq\pr^{-1}(A)\subseteq B_{0}.
\]
Now let $f:S^{n-1}\to\partial B_{1}=\ov B_{1}\setminus B_{1}$ be
a definable homeomorphism (whose existence follows by the hypothesis
that the ball is proper). By fixing base points, we can consider the
homotopy class $[f]$ as a non-zero element of $\pi_{n-1}^{\df}(B_{0}\setminus x)$
(namely $f$ is not definably homotopic to a constant in $B_{0}\setminus x$). 

\emph{A fortiori,} $0\neq[f]\in\pi_{n-1}^{\df}(\pr^{-1}(A)\setminus\pr^{-1}(y))$,
because if $f$ is homotopic to a constant within a smaller space,
it remains so in the larger space. Now observe that $\pr^{-1}(A)\setminus\pr^{-1}(y)=\pr^{-1}(A\setminus y)$
and by \prettyref{thm:B} we have $\pi_{n-1}^{\df}(\pr^{-1}(A\setminus y))\cong\pi_{n-1}(A\setminus y)$. 

We conclude that $\pi_{n-1}(A\setminus y)\neq0$, and since $A$ is
an open ball in the manifold $X/E$ this can happen only if $n\geq N$.
\end{proof}
So far we have not used the full strength of the assumption, namely
the topological compact domination. 
\begin{prop}
\label{prop:dimX-leq-dimXE}$\dim(X)\leq\dim_{\R}(X/E)$. 
\end{prop}
\begin{proof}
As before, let $n=\dim(X)$ and $N=\dim_{\R}(X/E)$. Let $A_{0}\subseteq X/E$
be an open $N$-ball, namely a set homeomorphic to $\{|x|\in\R^{N}\suchthat|x|<1\}$.
Let $A_{1}\subseteq A_{0}$ be the image of $\{|x|\in\R^{N}\suchthat|x|<1/2\}$
under the homeomorphism and note that $0\neq\pi_{N-1}(A_{0}\setminus\ov{A_{1}})$
and $A_{0}\setminus\ov A_{1}$ is a deformation retract of $A_{0}\setminus\{y\}$
for every $y\in A_{1}$. 

By \prettyref{thm:B}, we have $0\neq\pi_{N-1}(\pr^{-1}(A_{0}\setminus\ov{A_{1}}))^{\df}$,
so there is a map $f:S^{N-1}\to\pr^{-1}(A_{0}\setminus\ov{A_{1}})$
of pointed spaces with $0\neq[f]\in\pi_{N-1}(\pr^{-1}(A_{0}\setminus\ov{A_{1}}))^{\df}$. 

Since $A_{0}$ is a ball, we have $\pi_{N-1}(A_{0})=0$ and, by \prettyref{thm:B},
$\pi_{N-1}^{\df}(\pr^{-1}(A_{0}))=0$ as well. In particular $[f]=0$
when seen as an element of $\pi_{N-1}(\pr^{-1}(A_{0}))^{\df}$. This
is equivalent to say that $f$ can be extended to a definable map
$F:D\to\pr^{-1}(A_{0})$, where $D=S^{N-1}\times I$ and $F$ is a
definable homotopy (relative to the base point) between $f$ and a
constant map. 

Notice that $\dim(F(D))\leq\dim(D)=N$. Now assume for a contradiction
that $N<\dim(X)$. Then $\dim(F(D))<\dim(X)$, and therefore $F(D)$
has empty interior in $X$. By topological compact domination $(\pr\circ F)(D)$
has empty interior in $X/E$, so in particular there is some $y\in A_{1}$
such that $y\nin(\pr\circ F)(D)$.

It follows that the image of $F$ is disjoint from $\pr^{-1}(y)$,
namely $F$ takes values in $\pr^{-1}(A_{0})\setminus\pr^{-1}(y)=\pr^{-1}(A_{0}\setminus y)$
and witnesses the fact that $f$ is null-homotopic when seen as a
map into $\pr^{-1}(A_{0}\setminus y)$. 

We can now reach a contradiction as follows. Since $A_{0}\setminus\ov A_{1}$
is a deformation retract of $A_{0}\setminus\{y\}$, the inclusion
induces an isomorphism $\pi_{N-1}(A_{0}\setminus\ov A_{1})\cong\pi_{N-1}(A_{0}\setminus y)$.
By the functoriality part in \prettyref{thm:B}, there is an induced
isomorphism $\pi_{N-1}^{\df}(\pr^{-1}(A_{0}\setminus\ov A_{1}))\cong\pi_{N-1}^{\df}(\pr^{-1}(A_{0}\setminus y))$.
Moreover, this isomorphism sends the homotopy class of $f$ to the
homotopy class of $f$ itself, but seen as a map with a different
codomain. This is absurd since $f$ was not null-homotopic as a map
to $\pr^{-1}(A_{0}\setminus\ov A_{1})$, while we have shown that
it is null-homotopic as a map to $\pr^{-1}(A_{0}\setminus y)$.
\end{proof}
As a corollary we obtain. 
\begin{cor}[\cite{Edmundo2004}]
Let $G$ be an abelian definably compact and definably connected
group of dimension $n$. Then $\pi_{1}^{\df}(G)\cong\Z^{n}$ and $G[k]\cong(\Z/k\Z)^{n}$,
where $G[k]$ is the $k$-torsion subgroup. 
\end{cor}
\begin{proof}
By \cite{Berarducci2005}, $G/G^{00}$ is a compact abelian connected
Lie group and by the previous result its dimension is $n$. It follows
that $G/G^{00}$ is isomorphic to an $n$-dimensional torus, so $\pi_{1}(G/G^{00})\cong\Z^{n}$
and, by \prettyref{thm:B}, $\pi_{1}^{\df}(G)\cong\Z^{n}$ as well. 

To determine the $k$-torsion two approaches are possible. The first
is to argue as in \cite{Edmundo2004}, namely to observe that $G[k]\cong\pi_{1}^{\df}(G)/k\pi_{1}^{\df}(G)$
and $\pi_{1}^{\df}(G)\cong\Z^{n}$. Alternatively we can use the fact
that $G^{00}$ is divisible \cite{Berarducci2005} and torsion free
\cite{Hrushovski2007a}, so $G$ and $G/G^{00}$ have isomorphic torsion
subgroups. Since $G/G^{00}$ is a torus of dimension $n$, its torsion
is known and we obtain the desired result.
\end{proof}
Notice that in \cite{Edmundo2004} both the isomorphism $\pi_{1}^{\df}(G)\cong\Z^{n}$
and the determination of the $k$-torsion of $G$ is proved directly
without using $G/G^{00}$, while our argument is a reduction to the
case of the classical tori. 

\bibliographystyle{alpha}

\end{document}